\documentclass[11pt]{amsart}

\usepackage[marginparwidth=75pt]{geometry}

\usepackage{amsmath, amsthm, amssymb}

\usepackage{palatino}
\usepackage{enumitem}

\usepackage[abs]{overpic}
\usepackage{tikz}
\usetikzlibrary{calc, positioning}

\usepackage{xcolor}
\definecolor{indigo}{rgb}{0.29, 0.0, 0.51}
\definecolor{dred}{RGB}{237, 28, 36}
\usepackage[colorlinks, urlcolor=indigo, linkcolor=indigo, citecolor=indigo]{hyperref}

\theoremstyle{plain}
\newtheorem{theorem}{Theorem}[section]
\newtheorem{corollary}[theorem]{Corollary}
\newtheorem{proposition}[theorem]{Proposition}
\newtheorem{lemma}[theorem]{Lemma}

\theoremstyle{definition}
\newtheorem{definition}[theorem]{Definition}

\theoremstyle{remark}
\newtheorem{remark}[theorem]{Remark}
\newtheorem{example}[theorem]{Example}

\newcommand{\dfn}[1]{{\em #1}}
\newcommand{\R}{\mathbb{R}}
\newcommand{\Q}{\mathbb{Q}}
\newcommand{\Z}{\mathbb{Z}}

\DeclareMathOperator\tb{tb}
\DeclareMathOperator\rot{rot}

\begin{document}

\title{The Gompf $\theta$-Invariant of Canonical Contact Structures via Legendrian Surgery}

\author{Mohan Bhupal}
\address{Department of Mathematics, METU, Ankara, Turkey}
\email{bhupal@metu.edu.tr}

\author{Burak Ozbagci}
\address{Department of Mathematics, Ko\c{c} University, Istanbul, Turkey}
\email{bozbagci@ku.edu.tr}


\begin{abstract}
Let $\Gamma$ be a minimal connected negative-definite plumbing tree
with all vertices of genus zero, and let $Y_\Gamma$ be the oriented
link of the corresponding normal complex surface singularity, equipped
with its canonical contact structure $\xi_{\rm can}$. We give
an explicit Legendrian surgery description of $\xi_{\rm can}$,
showing that it is the unique consistent diagram-realizable contact
structure on $Y_\Gamma$, up to isomorphism.  We then derive a closed-form formula for
Gompf's $\theta$-invariant of $\xi_{\rm can}$ in the Seifert
fibered case, expressed purely in terms of the Hirzebruch--Jung
continued fraction expansions of the normalized Seifert invariants,
and prove a recursive leaf-to-root formula for arbitrary plumbing
trees. The Seifert formula recovers previously known formulas for
lens spaces, dihedral manifolds, and small Seifert fibered spaces
with complementary legs, and agrees with the N\'emethi--Nicolaescu
expression via the classical Hirzebruch--Zagier identity. As a final
application we show that $\xi_{\rm can}$ strictly minimizes
$\theta$ among all diagram-realizable contact structures on
$Y_\Gamma$, and we use this to rule out symplectic rational homology
ball fillings for a large class of Stein fillable contact rational
homology $3$-spheres.
\end{abstract}

\maketitle

\section{Introduction}\label{sec:intro}

Any minimal connected negative-definite plumbing tree $\Gamma$ with
all vertices of genus zero arises, by Grauert's contractibility
criterion~\cite{Grauert1962}, as the dual resolution graph of a normal
complex surface singularity. The oriented link $Y_\Gamma$ of this
singularity is a rational homology $3$-sphere carrying a canonical
contact structure $\xi_{\rm can}$ given by the field of
complex tangencies, uniquely determined up to isomorphism~\cite{CaubelNemethiPopescu-Pampu2006}.

In~\cite[Theorem~8.1]{BhupalOzbagci11}, we obtained an explicit Legendrian surgery
description of the canonical contact structure  in the
special case where $\Gamma$ is a star-shaped plumbing tree with three
legs, using the classification of fillable contact structures on the
Seifert fibered space $Y_\Gamma$. The first goal of the present paper
is to extend this description to arbitrary plumbing trees as above,
without relying on a classification of fillable contact structures, which
is unavailable in this generality.  Note, however,  that Cavallo and
Matkovi\v{c}~\cite{CavalloMatkovic2026b} have recently classified the
symplectically fillable contact structures on
negative-definite Seifert fibered spaces.

\begin{theorem}\label{thm:consistent-intro}
Let $\Gamma$ be a minimal connected negative-definite plumbing tree
with all vertices of genus zero. Then the consistent
diagram-realizable contact structure on $Y_\Gamma$ is the canonical
contact structure $\xi_{\rm can}$, up to isomorphism.
\end{theorem}

This surgery description has two consequences. The first is a
closed-form expression for Gompf's
$\theta$-invariant~\cite{Gompf98} of $\xi_{\rm can}$ in the
Seifert fibered case.

\begin{theorem}\label{thm:intro-main-seifert}
Let $Y=Y(e_0;r_1,\dots,r_k)$ be a Seifert fibered rational homology
$3$-sphere with normalized Seifert invariants $r_i\in\Q\cap(0,1)$,
central weight $e_0\le -1$, and negative rational Euler number
$e(Y)=e_0+r_1+\cdots+r_k<0$. For each $i=1,\dots,k$, write
\[
\frac{1}{r_i}=\frac{p_i}{q_i}=[a_1^i,\dots,a_{n_i}^i],
\qquad a_j^i\ge 2,
\]
where $p_i,q_i$ are coprime integers with $0<q_i<p_i$. Let $q_i^*$ be
the multiplicative inverse of $q_i$ modulo $p_i$, and set
$I(p_i/q_i)=\sum_{j=1}^{n_i}(a^i_j-3)$. Then
\[
\theta(\xi_{\rm can})
=
(2k-1)
-\sum_{i=1}^k\!\left(I(p_i/q_i)+\frac{q_i+q_i^*+2}{p_i}\right)
+
\frac{1}{e(Y)}\!\left(-e_0+k-2-\sum_{i=1}^k \frac{q_i+1}{p_i}\right)^{\!2}.
\]
\end{theorem}

The proof of Theorem~\ref{thm:intro-main-seifert} reduces
$\theta(\xi_{\rm can})$ to an inverse-intersection form computation
on the plumbing $4$-manifold. This is evaluated by combining the
tridiagonal matrix identity of Lemma~\ref{lem:leg} with the Schur
complement formula of Lemma~\ref{lem:schur-k}. The case $e_0\le -2$
proceeds directly from Gompf's handlebody formula via the Legendrian
surgery description of Theorem~\ref{thm:consistent-intro}; the case
$e_0=-1$, where the central vertex is not Legendrian-realizable,
follows from the same algebraic identity together with the
topological identity $\theta(\xi_{\rm can})=K_{\rm can}^2+\#V-2$
(Remark~\ref{rem:e0-minus-one}).
In Section~\ref{sec:recover} we show that
Theorem~\ref{thm:intro-main-seifert} specializes to the lens space
and dihedral formulas of
\cite[Propositions~9.3 and~9.8]{EtnyreOzbagciTosun2025} and to the
complementary-legs formula of
\cite[Proposition~4.4]{EtnyreOzbagciTosun2026}. In
Section~\ref{sec:nn} we show that after matching conventions, the
formula in Theorem~\ref{thm:intro-main-seifert} agrees with the
N\'emethi--Nicolaescu expression
\cite[Section~5.5]{NemethiNicolaescu2002} via the classical
Hirzebruch--Zagier identity, bypassing the Casson--Walker invariant
used in \cite{NemethiNicolaescu2002} and providing a self-contained
alternative.

\begin{remark}\label{rem:d-invariant}
Under the hypotheses of Theorem~\ref{thm:intro-main-seifert}, let
$\mathfrak{s}_{\rm can} \in \operatorname{Spin}^c(Y)$ be the
$\operatorname{Spin}^c$ structure induced by $\xi_{\rm can}$. Then the
Heegaard Floer correction term satisfies
\[
d\!\left(Y,\,\mathfrak{s}_{\rm can}\right)
= \frac{\theta(\xi_{\rm can}) + 2}{4},
\]
with $\theta(\xi_{\rm can})$ given explicitly by
Theorem~\ref{thm:intro-main-seifert}. Indeed, a star-shaped plumbing
tree has at most one bad vertex (the central one), hence is almost
rational; this is the special case of
Remark~\ref{rem:d-invariant-tree} where equality is guaranteed.
\end{remark}

For arbitrary negative-definite plumbing trees, the star-shaped decomposition
underlying Theorem~\ref{thm:intro-main-seifert} is no longer
available, and the Schur complement must be applied inductively
along $\Gamma$. Our second main result gives a recursive formula for
$\theta(\xi_{\rm can})$ that requires only a single leaf-to-root pass
on the vertex set $V(\Gamma)$.

\begin{theorem}\label{thm:intro-main-tree}
Let $\Gamma$ be a connected negative-definite tree with $N$
vertices, all genera zero, and weights $-a_v\le -1$. Set
$z_v:=a_v-2$, and fix any root $\rho\in V(\Gamma)$. Define
rational numbers $\alpha_v>0$, $s_v$, $\beta_v$ for each
$v\in V(\Gamma)$ by the leaf-to-root recursions
\[
\alpha_v=\frac{1}{a_v-\sum_{c\in C(v)}\alpha_c},
\qquad
s_v=z_v-\sum_{c\in C(v)}\beta_c,
\qquad
\beta_v=-\alpha_v s_v,
\]
where $C(v)$ denotes the set of children of $v$ (sums are empty when
$v$ is a leaf). Then
\[
\theta(\xi_{\rm can})
=
(N-2)-\sum_{v\in V(\Gamma)}\alpha_v s_v^2.
\]
The right-hand side is independent of the choice of root.
\end{theorem}

The quantity $\alpha_v$ is a generalized continued fraction of the
subtree rooted at $v$; when $\Gamma$ is star-shaped and $v$ is the
first vertex of the $i$-th leg, $\alpha_v$ specializes to the Seifert
invariant $r_i=q_i/p_i$.  

\begin{remark}\label{rem:intro-non-minimal-extension}
Theorem~\ref{thm:intro-main-seifert} is recovered as a special case
of Theorem~\ref{thm:intro-main-tree}
(Corollary~\ref{cor:star-recovery}).
\end{remark}

\begin{remark}\label{rem:d-invariant-tree}
Let $\Gamma$ be as in Theorem~\ref{thm:intro-main-tree}, let
$X_\Gamma$ be the associated plumbing $4$-manifold with intersection
form $Q_\Gamma$, and let $K_{\rm can}\in\operatorname{Char}(X_\Gamma)$
denote the canonical characteristic vector, defined by
$K_{\rm can}\cdot v + v\cdot v = -2$ for every $v\in V(\Gamma)$.
Khan~\cite{Khan2026} has recently proved N\'emethi's
conjecture~\cite{Nemethi2008} that
\[
d(Y_\Gamma,\mathfrak{s}_{\rm can})
\;=\;
\max_{k\in\mathfrak{s}_{\rm can}}\frac{k^2+N}{4}
\]
for every negative-definite plumbed rational homology 3-sphere.
Since $K_{\rm can}\in\mathfrak{s}_{\rm can}$, this recovers in
particular the Ozsv\'ath--Szab\'o plumbing
inequality~\cite[Theorem~9.6]{OzsvathSzabo2003abs}
\[
d(Y_\Gamma,\mathfrak{s}_{\rm can})
\;\geq\;
\frac{K_{\rm can}^2+N}{4}.
\]
Combining with Theorem~\ref{thm:intro-main-tree} and the identity
$\theta(\xi_{\rm can})=K_{\rm can}^2+N-2$
\cite[Remark~4.8]{NemethiNicolaescu2002} yields
\[
d(Y_\Gamma,\mathfrak{s}_{\rm can})
\;\geq\;
\frac{\theta(\xi_{\rm can})+2}{4}
\;=\;
\frac{N}{4}-\frac{1}{4}\sum_{v\in V(\Gamma)}\alpha_v s_v^2,
\]
with equality if and only if $K_{\rm can}$ realizes the maximum in
Khan's formula. By~\cite{Nemethi2005}, this is the case precisely
when $\Gamma$ is almost rational, equivalently when $\Gamma$ has at
most one bad vertex; the inequality can be strict otherwise.
Remark~\ref{rem:d-invariant} is the special case in which $\Gamma$ is
star-shaped, where almost-rationality is automatic (the central
vertex is the unique candidate for badness) and equality is
guaranteed.
\end{remark}

The second consequence of the surgery description, treated in
Section~\ref{sec:rhb}, is a strict minimization property of the
canonical contact structure among all diagram-realizable structures.

\begin{proposition}\label{prop:min}
Let $\Gamma$ be a minimal connected negative-definite plumbing tree
with all vertices of genus zero. Then
\[
\theta(\xi_{\rm can})<\theta(\xi)
\]
for any inconsistent diagram-realizable contact structure $\xi$ on
$Y_\Gamma$.
\end{proposition}

The proof is purely algebraic, based on the positivity of the entries
of $-Q_\Gamma^{-1}$ that follows from $-Q_\Gamma$ being an irreducible
Stieltjes matrix. The corresponding inequality among all
\emph{tight} contact structures on $Y_\Gamma$ was previously
established in~\cite[Proposition~1.7]{EtnyreOzbagciTosun2025} for star-shaped
plumbing trees with three legs, by reduction to the classification of
tight contact structures on Seifert fibered spaces with three
exceptional fibers. We do not require such a classification, and our
argument is therefore valid for arbitrary plumbing trees in our
class, but only in the diagram-realizable category.

As an application, recall that if $(Y,\xi)$ is an oriented rational
homology $3$-sphere admitting a symplectic rational homology ball
filling, then $\theta(\xi)=-2$. Combined with the inequality above
and the closed-form expressions of Theorems~\ref{thm:intro-main-seifert}
and~\ref{thm:intro-main-tree}, this rules out such fillings for a
large class of Stein fillable contact rational homology $3$-spheres
(Corollary~\ref{cor:S}).

\begin{remark}
Independently of the present work, Cavallo and
Matkovi\v{c}~\cite{CavalloMatkovic2026b} have recently classified the
symplectically fillable  contact structures on
negative-definite Seifert fibered spaces using Heegaard Floer
methods, and~\cite{CavalloMatkovic2026a} treats the opposite
orientation; they compute the maximal twisting number, a different
invariant from the one studied here.
\end{remark}

\medskip\noindent\textbf{Organization.}
Section~\ref{sec:mainsurgery} establishes the surgery description and
proves Theorem~\ref{thm:consistent-intro}. Section~\ref{sec:local} develops
a tridiagonal matrix identity and Section~\ref{sec:schur} the Schur
complement framework needed for the explicit formulas.
Section~\ref{sec:theta} proves
Theorem~\ref{thm:intro-main-seifert}.
Section~\ref{sec:recover} recovers known special cases, and
Section~\ref{sec:nn} compares with the N\'emethi--Nicolaescu formula.
Section~\ref{sec:general} develops the general tree theory and proves
Theorem~\ref{thm:intro-main-tree}. Section~\ref{sec:rhb} contains the
$\theta$-minimization (Proposition~\ref{prop:min}) and the application to symplectic rational
homology ball fillings. Section~\ref{sec:examples} contains examples of Seifert fibered spaces with $\theta(\xi_{\rm can})=-2$.

\section{Surgery diagram for the canonical contact structure}\label{sec:mainsurgery}

In this section we prove the surgery description of $\xi_{\rm can}$.
Throughout, $Y_\Gamma$ denotes the oriented link of the normal complex
surface singularity whose dual resolution graph is the minimal
connected negative-definite plumbing tree $\Gamma$ with all vertices
of genus zero.

\begin{definition}
A contact structure $\xi$ on $Y_\Gamma$ is \emph{diagram-realizable}
if it is obtained by Legendrian surgery on a Legendrian realization of
the unknots in the surgery diagram given by the plumbing tree
$\Gamma$. A diagram-realizable contact structure is called
\emph{consistent} if all stabilizations of all Legendrian unknots in
the diagram have the same sign; otherwise it is called
\emph{inconsistent}.
\end{definition}

We fix an orientation on each Legendrian unknot in the diagram so that
the linking number of the unknots corresponding to any two adjacent
vertices of $\Gamma$ is $+1$. Then $Y_\Gamma$ admits two consistent
contact structures: one obtained by Legendrian surgery on a diagram in
which every Legendrian unknot realizes its maximal rotation number,
and one in which every Legendrian unknot realizes its minimal rotation
number. These two contact structures are isomorphic. Our main result
of this section identifies their common isomorphism class with the
canonical contact structure on $Y_\Gamma$.

\begin{theorem}\label{thm:consistent}
The consistent diagram-realizable contact structure on $Y_\Gamma$ is
the canonical contact structure $\xi_{\rm can}$, up to
isomorphism.
\end{theorem}

\begin{proof}
Let $\pi\colon\tilde X\to X$ be the minimal good resolution of the
normal complex surface singularity $(X,0)$ whose resolution graph is
$\Gamma$, and let $E=\pi^{-1}(0)=\bigcup_{u\in V(\Gamma)} C_u$ be the
exceptional divisor, where each $C_u$ is a smooth rational curve of
self-intersection $-a_u\le -2$. Choose a strictly plurisubharmonic
function $\varphi\colon\tilde X\to\R$ defined on a neighborhood of
$E$ with $\varphi^{-1}(0)=E$, and let
\[
P_\Gamma:=\{\,p\in\tilde X:\varphi(p)\le\epsilon\,\}
\]
for $\epsilon>0$ sufficiently small. Then $P_\Gamma$ is a compact
complex manifold with strictly pseudoconvex boundary, diffeomorphic
to the plumbing $4$-manifold associated to $\Gamma$, and the field
of complex tangencies along $Y_\Gamma:=\partial P_\Gamma$ is, by
\cite{CaubelNemethiPopescu-Pampu2006}, the canonical contact
structure $\xi_{\rm can}$.

By Bogomolov--de Oliveira~\cite{BogomolovDeOliveira97}, the complex
structure on $P_\Gamma$ can be deformed, through strictly
pseudoconvex complex structures with the same boundary contact
structure, to a Stein structure $J_\Gamma$ on the underlying smooth
$4$-manifold. In particular, $(P_\Gamma,J_\Gamma)$ is a Stein filling
of $(Y_\Gamma,\xi_{\rm can})$.

By Eliashberg's theorem~\cite{Eliashberg90a} (see also
\cite[Theorem~1.3]{Gompf98}), the Stein manifold $(P_\Gamma,J_\Gamma)$
admits a Weinstein handle decomposition consisting of a single
$0$-handle and one $2$-handle $H_u$ for each vertex $u\in V(\Gamma)$,
where $H_u$ is attached along a Legendrian unknot
$L_u\subset(S^3,\xi_{\rm std})$ with smooth framing
$\tb(L_u)-1=-a_u$. Hence $\tb(L_u)=1-a_u$ for every $u\in V(\Gamma)$.
Moreover, the smooth attaching link can be arranged so that the
linking numbers between Legendrian unknots of adjacent vertices are
$+1$ and zero otherwise, recovering the standard plumbing
presentation.

Let
\[
\Sigma_u=D_u\cup_{L_u} D_{c,u}
\]
be the closed surface obtained by gluing a Seifert disk
$D_u\subset B^4$ for $L_u$ to the core disk $D_{c,u}$ of $H_u$. In
$H_2(P_\Gamma;\Z)$, the class $[\Sigma_u]$ agrees with $[C_u]$, since
both cap the core disk $D_{c,u}$ with a disk in the $0$-handle, and
any two such disks are homotopic rel boundary.

By Gompf's first Chern class formula~\cite[Proposition~2.3]{Gompf98},
\[
\rot(L_u)
=\langle c_1(P_\Gamma,J_\Gamma),[\Sigma_u]\rangle
=\langle c_1(P_\Gamma,J_\Gamma),[C_u]\rangle.
\]

Since the Bogomolov--de Oliveira deformation is a smooth family of
complex (in particular almost complex) structures on $P_\Gamma$, the
first Chern class
$c_1(P_\Gamma,J_\Gamma)=c_1(P_\Gamma,J_0)\in H^2(P_\Gamma;\Z)$ is
independent of the parameter, where $J_0$ is the original complex
structure on the resolution neighborhood. Applying adjunction to the
smooth rational curve $C_u\subset\tilde X$, we obtain
\[
\langle c_1(P_\Gamma,J_\Gamma),[C_u]\rangle
=\langle c_1(P_\Gamma,J_0),[C_u]\rangle
=[C_u]^2+2=-a_u+2,
\]
and hence $\rot(L_u)=-(\tb(L_u)+1)=-(a_u-2)$.

Since the orientation of each $C_u$ is induced by the holomorphic
structure $J_0$ uniformly, the sign is the same across all vertices.
Therefore the Legendrian surgery diagram associated to
$(P_\Gamma,J_\Gamma)$ is consistent: every $L_u$ realizes its minimal
rotation number simultaneously.

There are exactly two consistent Legendrian realizations of $\Gamma$:
the all-minimal one above and the all-maximal one
$\rot(L_u)=+(a_u-2)$. Reversing the orientation of every Legendrian
unknot $L_u$ simultaneously preserves $\tb(L_u)$, flips the sign of
$\rot(L_u)$, and preserves all pairwise linking numbers. Since
contact surgery depends only on the unoriented Legendrian attaching
locus together with the contact framing, the all-minimal and
all-maximal diagrams yield isomorphic contact structures on
$Y_\Gamma$. We conclude that the consistent diagram-realizable
contact structure on $Y_\Gamma$ is $\xi_{\rm can}$, up to
isomorphism.
\end{proof}

\section{A tridiagonal matrix identity}\label{sec:local}

Suppose that $0<q<p$ are coprime integers and let
\[
\frac pq=[a_1,\dots,a_n]
=
a_1-\cfrac{1}{a_2-\cfrac{1}{\cdots-\cfrac{1}{a_n}}}
\]
be the Hirzebruch--Jung continued fraction expansion of $p/q$, with
$a_i\ge 2$ for all $i$. We set
\[
I\!\left(\frac pq\right)=\sum_{j=1}^n(a_j-3).
\]

\begin{definition}
Let $a_1,\ldots,a_n$ be a sequence of integers with $a_i\ge 2$. The
tridiagonal matrix associated to $a_1,\ldots,a_n$ is
\[
A=
\begin{pmatrix}
-a_1 & 1 & 0 & \cdots & 0\\
1 & -a_2 & 1 & \ddots & \vdots\\
0 & 1 & -a_3 & \ddots & 0\\
\vdots & \ddots & \ddots & \ddots & 1\\
0 & \cdots & 0 & 1 & -a_n
\end{pmatrix}.
\]
For $n=1$, this means $A=[-a_1]$.
\end{definition}

\begin{lemma}\label{lem:leg}
Suppose that $0<q<p$ are coprime integers and let
$\frac pq=[a_1,\dots,a_n]$ be the Hirzebruch--Jung continued fraction
expansion of $p/q$, with $a_i\ge 2$ for all $i$. Let $A$ be the
tridiagonal matrix associated to $a_1,\ldots,a_n$, and set
\[
{\bf w}=(a_1-2,\dots,a_n-2)^T,
\qquad
{\bf u}=(1,0,\dots,0)^T\in\R^n.
\]
If $q^*$ denotes the multiplicative inverse of $q$ modulo $p$, then
\[
{\bf w}^TA^{-1}{\bf w}
=
2-I(p/q)-n-\frac{q+q^*+2}{p},
\]
and
\[
{\bf u}^TA^{-1}{\bf w}
=
\frac{q+1-p}{p}.
\]
\end{lemma}

\begin{proof}
If $n=1$, the statement is immediate from the explicit formulas
\[
A=[-a_1],\qquad p=a_1,\qquad q=q^*=1,\qquad
{\bf w}=(a_1-2),\qquad {\bf u}=(1).
\]
So assume $n\ge 2$.

Let ${\bf c}=A^{-1}{\bf w}=(c_1,\dots,c_n)^T$. Since
$A{\bf c}={\bf w}$, summing the resulting linear equations gives
\[
-\sum_{i=1}^n(a_i-2)c_i-c_1-c_n
=
\sum_{i=1}^n(a_i-2).
\]
Therefore
\[
{\bf w}^TA^{-1}{\bf w}
=
-\sum_{i=1}^n(a_i-2)-c_1-c_n.
\]
Since $\sum_{i=1}^n(a_i-2)=I(p/q)+n$, we obtain
\[
{\bf w}^TA^{-1}{\bf w}
=
-I(p/q)-n-c_1-c_n.
\]

By Cramer's rule,
\[
c_1=\frac{\det A_1}{\det A},
\]
where $A_1$ is obtained from $A$ by replacing the first column by
${\bf w}$. Using the standard determinant formulas,
\[
\det A=(-1)^n p,
\qquad
\det A_1=(-1)^{n+1}(p-q-1),
\]
and hence
\[
c_1=\frac{q+1-p}{p}.
\]

Applying the same argument to the reversed chain
$[a_n,\dots,a_1]=p/q^*$ gives
\[
c_n=\frac{q^*+1-p}{p}.
\]
Substituting into the previous formula yields
\[
{\bf w}^TA^{-1}{\bf w}
=
2-I(p/q)-n-\frac{q+q^*+2}{p}.
\]

Finally,
\[
{\bf u}^TA^{-1}{\bf w}={\bf u}^T{\bf c}=c_1=\frac{q+1-p}{p}.
\]
\end{proof}

\section{The Schur complement formula}\label{sec:schur}

Let $k\geq 3$ be an integer and let $Y=Y(e_0;r_1,\ldots,r_k)$ be a
Seifert fibered space in standard form, where $e_0\in\Z$ and
$r_i\in\Q\cap(0,1)$. For $1\leq i\leq k$, let
\[
\frac{1}{r_i}=[a_1^i,\dots,a_{n_i}^i],
\qquad a_j^i\ge 2,
\]
be the Hirzebruch--Jung continued fraction expansion of $1/r_i$. Let
$\Gamma$ denote the standard star-shaped plumbing graph associated to
$Y$: the central vertex has weight $e_0$, and the $i$-th leg is the
linear chain with weights $-a_1^i,\dots,-a_{n_i}^i$. Let $Q_\Gamma$
denote the intersection matrix of $\Gamma$. For the $i$-th leg, let
$A_i$ denote the tridiagonal intersection matrix
\[
A_i=
\begin{pmatrix}
-a_1^i & 1 & 0 & \cdots & 0\\
1 & -a_2^i & 1 & \ddots & \vdots\\
0 & 1 & -a_3^i & \ddots & 0\\
\vdots & \ddots & \ddots & \ddots & 1\\
0 & \cdots & 0 & 1 & -a_{n_i}^i
\end{pmatrix}
\]
associated to $a_1^i,\dots,a_{n_i}^i$.

\begin{lemma}[Schur complement formula]\label{lem:schur-k}
Let $\Gamma$ be a star-shaped plumbing graph with $k$ legs, as
described above. Set
\[
A=A_1\oplus\cdots\oplus A_k,
\qquad
{\bf u}_i=(1,0,\dots,0)^T\in\R^{n_i},
\qquad
U=({\bf u}_1^T,\dots,{\bf u}_k^T).
\]
Then the following holds.
\begin{enumerate}[label=(\roman*)]
\item The intersection matrix of $\Gamma$ has the block form
\[
Q_\Gamma=
\begin{pmatrix}
e_0 & U\\
U^T & A
\end{pmatrix}.
\]

\item $Q_\Gamma$ is invertible if and only if
\[
\Delta:=e_0-UA^{-1}U^T\neq 0,
\]
and in that case
\[
Q_\Gamma^{-1}
=
\begin{pmatrix}
\Delta^{-1} & -\Delta^{-1}UA^{-1}\\
-A^{-1}U^T\Delta^{-1} & A^{-1}+A^{-1}U^T\Delta^{-1}UA^{-1}
\end{pmatrix}.
\]

\item For any vector
\[
\mathbf z=
\begin{pmatrix}
z_0\\
\mathbf z_1\\
\vdots\\
\mathbf z_k
\end{pmatrix}\in\R^{1+n_1+\cdots+n_k},
\]
where $\mathbf z_i\in\R^{n_i}$, one has
\[
\mathbf z^TQ_\Gamma^{-1}\mathbf z
=
\sum_{i=1}^k \mathbf z_i^TA_i^{-1}\mathbf z_i
+
\frac{1}{\Delta}
\!\left(z_0-\sum_{i=1}^k {\bf u}_i^TA_i^{-1}\mathbf z_i\right)^{\!2}.
\]

\item
$
\Delta=e(Y):=e_0+\sum_{i=1}^k r_i.
$
\end{enumerate}
\end{lemma}

\begin{proof}
We orient the knots in the surgery diagram corresponding to the
plumbing graph $\Gamma$ so that the linking number of any two knots
corresponding to adjacent vertices is $+1$. Therefore,
\[
Q_\Gamma=
\begin{pmatrix}
e_0 & U\\
U^T & A
\end{pmatrix}.
\]
Since $A=A_1\oplus\cdots\oplus A_k$ is invertible, the Schur
complement of $A$ in $Q_\Gamma$ is $\Delta=e_0-UA^{-1}U^T$. It follows
that $Q_\Gamma$ is invertible if and only if $\Delta\neq 0$, and in
that case
\[
Q^{-1}_\Gamma=
\begin{pmatrix}
\Delta^{-1} & -\Delta^{-1}UA^{-1}\\
-A^{-1}U^T\Delta^{-1} & A^{-1}+A^{-1}U^T\Delta^{-1}UA^{-1}
\end{pmatrix}.
\]

Write $\mathbf z=\begin{pmatrix}z_0\\\mathbf z_{\rm leg}\end{pmatrix}$,
where $\mathbf z_{\rm leg}=(\mathbf z_1,\dots,\mathbf z_k)^T$. Then
\[
\mathbf z^TQ_\Gamma^{-1}\mathbf z
=
z_0^2\Delta^{-1}
-2z_0\Delta^{-1}UA^{-1}\mathbf z_{\rm leg}
+\mathbf z_{\rm leg}^TA^{-1}\mathbf z_{\rm leg}
+\mathbf z_{\rm leg}^TA^{-1}U^T\Delta^{-1}UA^{-1}\mathbf z_{\rm leg}.
\]

Since $A$ is block diagonal,
\[
\mathbf z_{\rm leg}^TA^{-1}\mathbf z_{\rm leg}
=
\sum_{i=1}^k \mathbf z_i^TA_i^{-1}\mathbf z_i,
\]
and
$UA^{-1}\mathbf z_{\rm leg}=\sum_{i=1}^k {\bf u}_i^TA_i^{-1}\mathbf z_i$.
Hence
\[
\mathbf z_{\rm leg}^TA^{-1}U^T\Delta^{-1}UA^{-1}\mathbf z_{\rm leg}
=
\Delta^{-1}\!\left(\sum_{i=1}^k {\bf u}_i^TA_i^{-1}\mathbf z_i\right)^{\!2}.
\]

Substituting and grouping terms involving $\Delta^{-1}$ gives
\[
\mathbf z^TQ_\Gamma^{-1}\mathbf z
=
\sum_{i=1}^k \mathbf z_i^TA_i^{-1}\mathbf z_i
+
\Delta^{-1}\!\left(z_0-\sum_{i=1}^k {\bf u}_i^TA_i^{-1}\mathbf z_i\right)^{\!2}.
\]

Finally, since $(A_i^{-1})_{11}=-r_i$, we have
\[
UA^{-1}U^T=\sum_{i=1}^k (A_i^{-1})_{11}=-\sum_{i=1}^k r_i,
\]
and therefore $\Delta=e_0+\sum_{i=1}^k r_i$.
\end{proof}

\section{The Seifert fibered case}\label{sec:theta}

Let $Y$ be an oriented rational homology $3$-sphere and let $\xi$ be
an oriented tangent $2$-plane field on $Y$. Following
Gompf~\cite{Gompf98}, the homotopy class of $\xi$ is determined by the
$spin^c$ structure induced by $\xi$ together with a rational invariant
$\theta(\xi)$, defined as follows. Given $(Y,\xi)$, there exists a
compact smooth $4$-manifold $X$ equipped with an almost complex
structure $J$ such that $\partial X=Y$ and the complex tangencies
$TY\cap JTY$ are homotopic to $\xi$ as oriented $2$-plane fields on
$Y$. Denote by $\chi(X)$ and $\sigma(X)$ the Euler characteristic and
signature of $X$, respectively, and let $c_1(X,J)$ denote the first
Chern class of the almost complex manifold $(X,J)$. Gompf showed that
the quantity
\[
c_1^2(X,J)-2\chi(X)-3\sigma(X)
\]
is independent of the choice of $(X,J)$ satisfying the conditions
above. Hence it defines an invariant of the homotopy class of $\xi$,
denoted by $\theta(\xi)$. Since $c_1^2(X,J)$ is defined using rational
coefficients in (co)homology, the invariant $\theta(\xi)$ takes
values in $\Q$. Moreover, $\theta(\xi)$ depends only on $Y$ and the
homotopy class of $\xi$, is independent of the orientation of $\xi$,
and changes sign when the orientation of $Y$ is
reversed~\cite[Theorem~4.5]{Gompf98}. It is well-known that if
$(Y,\xi)$ admits a symplectic rational homology ball filling, then
$\theta(\xi)=-2$.

For coprime integers $0<q<p$ with Hirzebruch--Jung continued fraction
$p/q=[a_1,\dots,a_n]$, $a_i\ge 2$, let $q^*$ denote the multiplicative
inverse of $q$ modulo $p$, and set $I(p/q):=\sum_j(a_j-3)$.

\begin{theorem}\label{thm:main-seifert}
Let $Y=Y(e_0;r_1,\dots,r_k)$ be a Seifert fibered space with
normalized invariants $r_i\in\Q\cap(0,1)$, central weight
$e_0\le -1$, and negative rational Euler number
\[
e(Y)=e_0+r_1+\cdots+r_k<0.
\]
For each $i=1,\dots,k$, write
\[
\frac{1}{r_i}=\frac{p_i}{q_i}=[a_1^i,\dots,a_{n_i}^i],
\qquad a_j^i\ge 2,
\]
where $p_i,q_i$ are coprime integers with $0<q_i<p_i$. Then the
canonical contact structure $\xi_{\rm can}$ on $Y$ satisfies
\[
\theta(\xi_{\rm can})
=
(2k-1)
-\sum_{i=1}^k\!\left(I(p_i/q_i)+\frac{q_i+q_i^*+2}{p_i}\right)
+\frac{1}{e(Y)}\!\left(-e_0+k-2-\sum_{i=1}^k\frac{q_i+1}{p_i}\right)^{\!2}.
\]
\end{theorem}

\begin{proof}
Let $\Gamma$ be the negative-definite star-shaped plumbing graph
corresponding to $Y=Y(e_0;r_1,\dots,r_k)$, with central vertex of
weight $e_0\le -2$, and with the $i$-th leg determined by
$1/r_i=p_i/q_i=[a_1^i,\dots,a_{n_i}^i]$. Let $X_\Gamma$ be the
corresponding plumbing $4$-manifold, and let $Q_\Gamma$ denote its
intersection matrix. For each $i$, let $A_i$ be the tridiagonal
matrix associated to $a_1^i,\dots,a_{n_i}^i$, and define
\[
{\bf w}_i=(a_1^i-2,\dots,a_{n_i}^i-2)^T,
\qquad
{\bf u}_i=(1,0,\dots,0)^T.
\]
Set
\[
A=A_1\oplus\cdots\oplus A_k,\qquad
{\bf w}=
\begin{pmatrix}
{\bf w}_1\\
\vdots\\
{\bf w}_k
\end{pmatrix},\qquad
U=({\bf u}_1^T,\dots,{\bf u}_k^T).
\]
Then
\[
Q_\Gamma=
\begin{pmatrix}
e_0 & U\\
U^T & A
\end{pmatrix}.
\]

By Theorem~\ref{thm:consistent}, the canonical contact structure
$\xi_{\rm can}$ on $Y$ is realized as the boundary of the Stein
structure on $X_\Gamma$ obtained by Legendrian surgery along the
consistent (all-maximal-rotation) diagram associated to $\Gamma$.
With orientations chosen so that all pairwise linking numbers between
adjacent unknots are $+1$, the rotation vector is
\[
\mathbf{z}=\begin{pmatrix}-e_0-2\\\mathbf{w}\end{pmatrix}.
\]
By Gompf's formula for Stein
handlebodies~\cite[Proposition~2.3]{Gompf98},
\[
c_1^2(X_\Gamma,J)=\mathbf{z}^TQ_\Gamma^{-1}\mathbf{z}.
\]

Since $\Gamma$ has $1+\sum_{i=1}^k n_i$ vertices,
\[
\chi(X_\Gamma)=2+\sum_{i=1}^k n_i,\qquad
\sigma(X_\Gamma)=-(1+\sum_{i=1}^k n_i),
\]
and hence
\[
-2\chi(X_\Gamma)-3\sigma(X_\Gamma)=-1+\sum_{i=1}^k n_i.
\]

By Lemma~\ref{lem:schur-k},
\[
c_1^2(X_\Gamma,J)
=
\sum_{i=1}^k {\bf w}_i^TA_i^{-1}{\bf w}_i
+
\frac{1}{\Delta}\!\left(z_0-\sum_{i=1}^k {\bf u}_i^TA_i^{-1}{\bf w}_i\right)^{\!2},
\]
where $\Delta=e(Y)$.

By Lemma~\ref{lem:leg}, for each $i$,
\[
{\bf w}_i^TA_i^{-1}{\bf w}_i
=
2-I(p_i/q_i)-n_i-\frac{q_i+q_i^*+2}{p_i},
\qquad
{\bf u}_i^TA_i^{-1}{\bf w}_i=\frac{q_i+1-p_i}{p_i}.
\]
Substituting,
\[
c_1^2(X_\Gamma,J)
=
\sum_{i=1}^k\!\left(2-I(p_i/q_i)-n_i-\frac{q_i+q_i^*+2}{p_i}\right)
+
\frac{1}{\Delta}\!\left(-e_0-2-\sum_{i=1}^k\frac{q_i+1-p_i}{p_i}\right)^{\!2}.
\]

Now
\[
\sum_{i=1}^k\frac{q_i+1-p_i}{p_i}
=
\sum_{i=1}^k\!\left(\frac{q_i+1}{p_i}-1\right)
=
\sum_{i=1}^k\frac{q_i+1}{p_i}-k,
\]
so
\[
-e_0-2-\sum_{i=1}^k\frac{q_i+1-p_i}{p_i}
=
-e_0+k-2-\sum_{i=1}^k\frac{q_i+1}{p_i}.
\]
Hence
\[
c_1^2(X_\Gamma,J)
=
\sum_{i=1}^k\!\left(2-I(p_i/q_i)-n_i-\frac{q_i+q_i^*+2}{p_i}\right)
+
\frac{1}{\Delta}\!\left(-e_0+k-2-\sum_{i=1}^k\frac{q_i+1}{p_i}\right)^{\!2}.
\]

Combining with $-2\chi(X_\Gamma)-3\sigma(X_\Gamma)=-1+\sum_{i=1}^k n_i$,
we obtain
\[
\theta(\xi_{\rm can})
=
\sum_{i=1}^k\!\left(2-I(p_i/q_i)-n_i-\frac{q_i+q_i^*+2}{p_i}\right)
-1+\sum_{i=1}^k n_i
+
\frac{1}{\Delta}\!\left(-e_0+k-2-\sum_{i=1}^k\frac{q_i+1}{p_i}\right)^{\!2}.
\]
The terms involving $n_i$ cancel, and replacing $\Delta$ with $e(Y)$
yields the stated formula.

The argument above assumed $e_0\le -2$, so that the central vertex
is Legendrian-realizable and Theorem~\ref{thm:consistent} identifies
$\xi_{\rm can}$ with the boundary contact structure of the Stein
structure on $X_\Gamma$. When $e_0=-1$, this direct surgery
interpretation fails, but the formula nonetheless holds by
Remark~\ref{rem:e0-minus-one}.
\end{proof}

\begin{remark}\label{rem:e0-minus-one}
The case $e_0=-1$ of Theorem~\ref{thm:main-seifert} requires separate
justification, since the central vertex of $\Gamma$ cannot then be
Legendrian-realized in $(S^3,\xi_{\rm std})$ (Legendrian unknots
there have $\tb\le -1$), so Theorem~\ref{thm:consistent} does not
directly apply. The formula nevertheless holds: the algebraic
identity $z^T Q_\Gamma^{-1}z = K_{\rm can}^2$ shows that the
right-hand side of Theorem~\ref{thm:main-seifert} equals
$K_{\rm can}^2+\#V-2$, and this equals $\theta(\xi_{\rm can})$ on $Y$
by~\cite[Remark~4.8]{NemethiNicolaescu2002}. Equivalently,
Proposition~\ref{prop:nemethi-agreement} expresses the same identity
via the Hirzebruch--Zagier formula and the
N\'emethi--Nicolaescu expression for $K_{\rm can}^2+\#V$, neither of
which requires $e_0\le -2$.
\end{remark}

\section{Recovering previous results}\label{sec:recover}

In this section we show that Theorem~\ref{thm:main-seifert}
specializes to three known formulas: the lens space formula of
\cite[Prop.~9.3]{EtnyreOzbagciTosun2025}, the dihedral formula of
\cite[Prop.~9.8]{EtnyreOzbagciTosun2025}, and the complementary-legs
formula of \cite[Prop.~4.4]{EtnyreOzbagciTosun2026}.

\subsection{Lens spaces}
Let $0<q<p$ be coprime integers with $p/q=[a_1,\dots,a_n]$, and let
$L(p,q)$ denote the oriented lens space, given by the linear plumbing
with weights $-a_1,\dots,-a_n$. The canonical contact structure
$\xi_{\rm can}$ on $L(p,q)$ arises from Legendrian surgery on this
chain with each unknot at maximal rotation number, giving rotation
vector $\mathbf{w}=(a_1-2,\dots,a_n-2)^T$.

\begin{proposition}
With the notation above,
\[
\theta(\xi_{\rm can})
=
-\,I\!\left(\frac pq\right)-\frac{q+q^*+2}{p},
\]
recovering \cite[Prop.~9.3]{EtnyreOzbagciTosun2025}.
\end{proposition}

\begin{proof}
The Stein filling has $\chi=1+n$ and $\sigma=-n$, so
$\theta(\xi_{\rm can})=c_1^2+n-2$ with
$c_1^2=\mathbf{w}^TA^{-1}\mathbf{w}$. By Lemma~\ref{lem:leg},
$c_1^2=2-I(p/q)-n-(q+q^*+2)/p$, and the stated formula follows.
\end{proof}

\subsection{Dihedral $3$-manifolds}

We show that for $k=3$ and $r_1=r_3=1/2$, Theorem~\ref{thm:main-seifert}
reduces to the formula of \cite[Prop.~9.8]{EtnyreOzbagciTosun2025}
for the dihedral spherical manifold $D(p,q)$.

\begin{proposition}\label{prop:dihedral-eot}
Let $Y=Y(e_0;1/2,r,1/2)$ be a Seifert fibered space in standard form
with $e_0\le -2$. Write $1/r=[a_1,\dots,a_n]$ and set
$p/q:=[-e_0,a_1,\dots,a_n]$. Then
\[
\theta(\xi_{\rm can})
=
1-I\!\left(\frac{p}{q}\right)-\frac{q^*}{p-q},
\]
where $q^*$ is the inverse of $q$ modulo $p-q$.
\end{proposition}

\begin{proof}
For each of the two legs with $r_i=1/2$, we have $p_i=2$,
$q_i=q_i^*=1$, so $I(2)+(q_i+q_i^*+2)/p_i=1$ and $(q_i+1)/p_i=1$.
Writing $1/r=p_0/q_0=[a_1,\dots,a_n]$, Theorem~\ref{thm:main-seifert}
with $k=3$ gives
\[
\theta(\xi_{\rm can})
=
3-I(p_0/q_0)-\frac{q_0+q_0^*+2}{p_0}
+\frac{(p_0(e_0+1)+q_0+1)^2}{p_0(p_0(e_0+1)+q_0)}.
\tag{1}
\]

Now $p/q=-e_0-q_0/p_0=(-e_0p_0-q_0)/p_0$, so $q=p_0$,
$p=-e_0p_0-q_0$, and $q_0=-e_0q-p$. Since $p\equiv -q_0\pmod q$, the
inverse $p^*$ of $p$ modulo $q$ satisfies $q_0^*=q-p^*$. A direct
substitution then yields
\[
p_0(e_0+1)+q_0=-(p-q),\qquad
\frac{q_0+q_0^*+2}{p_0}=-e_0+1-\frac{p+p^*-2}{q},
\]
\[
I(p/q)=(-e_0-3)+I(p_0/q_0).
\]
Inserting these into~(1) gives
\[
\theta(\xi_{\rm can})
=
-1-I\!\left(\frac{p}{q}\right)
+\frac{p+p^*-2}{q}
-\frac{(p-q-1)^2}{q(p-q)}.
\tag{2}
\]

It remains to combine the two rational terms in~(2). Writing
$(p+p^*-2)/q=1+(p-q+p^*-2)/q$ and combining over the common
denominator $q(p-q)$, a direct expansion shows
$(p-q)(p-q+p^*-2)-(p-q-1)^2=(p-q)p^*-1$, hence
\[
\theta(\xi_{\rm can})
=
-I\!\left(\frac{p}{q}\right)+\frac{m}{p-q},
\qquad
m:=\frac{(p-q)p^*-1}{q}.
\]
Since $pp^*\equiv 1\pmod q$ and $p\equiv p-q\pmod q$, we have
$(p-q)p^*\equiv 1\pmod q$, so $m\in\Z$. From $(p-q)p^*-qm=1$,
reducing $\!\!\!\mod p-q$ gives $qm\equiv -1\pmod{p-q}$, so
$m\equiv -q^*\pmod{p-q}$. The bounds $1\le p^*\le q-1$ force
$0<m<p-q$, hence $m=(p-q)-q^*$ and $m/(p-q)=1-q^*/(p-q)$, yielding
the stated formula.
\end{proof}

\begin{remark}\label{rem:dihedral-eot}
In the notation of \cite{EtnyreOzbagciTosun2025},
$p/q=[-e_0,a_1,\dots,a_n]$ is the continued fraction of the dihedral
resolution graph, and $q^*/(p-q)=1/[a_n,a_{n-1},\dots,a_1,-e_0-1]$, so
Proposition~\ref{prop:dihedral-eot} matches
\cite[Prop.~9.8]{EtnyreOzbagciTosun2025} on the nose.
\end{remark}

\subsection{Small Seifert fibered spaces with complementary legs}
Finally, we show that for $k=3$ and $r_1+r_3=1$,
Theorem~\ref{thm:main-seifert} reduces to
\cite[Prop.~4.4]{EtnyreOzbagciTosun2026}.

\begin{proposition}\label{prop:general-to-EOT-complementary}
Assume $Y=Y(e_0;r_1,r_2,r_3)$ is in standard form with $e_0\le -2$
and $r_1+r_3=1$. Write
\[
\frac{1}{r_1}=\frac{p_1}{q_1},\qquad
\frac{1}{r_2}=\frac{p_2}{q_2}=[a^2_1,\dots,a^2_{n_2}],\qquad
a^2_0:=-e_0,
\]
and set $p/q:=[a^2_0,a^2_1,\dots,a^2_{n_2}]$. Then
\[
\theta(\xi_{\rm can})
=
1-I(p/q)
-\frac{1}{[a^2_{n_2},\dots,a^2_1,a^2_0-1]}
+\frac{2(p_1-2)}{p_1(p-q)}
-\frac{(p_1-2)^2q}{p_1^2(p-q)},
\]
matching \cite[Prop.~4.4]{EtnyreOzbagciTosun2026} after replacing
$p_1,q_1$ by $\tilde{p},\tilde{q}$.
\end{proposition}

\begin{proof}
Since $r_1+r_3=1$, set $r_1=q_1/p_1$ and $r_3=(p_1-q_1)/p_1$. The
identities $I(p_1/q_1)+I(p_1/(p_1-q_1))=-2$ and
$(p_1-q_1)^*=p_1-q_1^*$ give
\[
\sum_{i=1,3}\!\left(I(p_i/q_i)+\frac{q_i+q_i^*+2}{p_i}\right)=\frac{4}{p_1},
\qquad
\sum_{i=1,3}\frac{q_i+1}{p_i}=1+\frac{2}{p_1},
\]
and $\Delta=e_0+1+q_2/p_2$. Hence Theorem~\ref{thm:main-seifert} gives
\[
\theta(\xi_{\rm can})
=
5-\frac{4}{p_1}
-I(p_2/q_2)-\frac{q_2+q_2^*+2}{p_2}
+\frac{1}{e_0+1+q_2/p_2}
\!\left(-e_0-\frac{2}{p_1}-\frac{q_2+1}{p_2}\right)^{\!2}.
\tag{1}
\]

From $p/q=[-e_0,a^2_1,\dots,a^2_{n_2}]=(-e_0p_2-q_2)/p_2$ we obtain
$q=p_2$, $p=-e_0q-q_2$, $q_2=-e_0q-p$, and
$I(p/q)=(-e_0-3)+I(p_2/q_2)$. With $p^*$ the inverse of $p$ modulo
$q$, we have $q_2^*=q-p^*$, and therefore
\[
\frac{q_2+q_2^*+2}{p_2}=-e_0+1-\frac{p+p^*-2}{q},\qquad
e_0+1+\frac{q_2}{p_2}=-\frac{p-q}{q},
\]
\[
-e_0-\frac{2}{p_1}-\frac{q_2+1}{p_2}=\frac{p-1}{q}-\frac{2}{p_1}.
\]
Substituting into~(1),
\[
\theta(\xi_{\rm can})
=
1-I(p/q)-\frac{4}{p_1}
+\frac{p+p^*-2}{q}
-\frac{(p_1(p-1)-2q)^2}{p_1^2\,q(p-q)}.
\tag{2}
\]

We now simplify the three trailing terms in~(2). Combining over
$p_1^2\,q(p-q)$ and expanding $(p_1(p-1)-2q)^2$ gives a numerator
\[
p_1^2\!\left[(p+p^*-2)(p-q)-(p-1)^2\right]+4p_1\,q(q-1)-4q^2.
\]
Set $m:=(pp^*-1)/q\in\Z$ and $r:=p+p^*-q-m$. Using $pp^*=qm+1$, the
bracketed expression simplifies to $q(m-p-p^*+2)=q(2-q-r)$, so the
numerator equals
\[
q\!\left[-p_1^2\,r+p_1^2(2-q)+4p_1(q-1)-4q\right]
=
q\!\left[-p_1^2\,r+(p_1-2)\bigl(p_1(2-q)+2q\bigr)\right].
\]
One checks that $r$ is the inverse of $q$ modulo $p-q$, whence
$r/(p-q)=1/[a^2_{n_2},\dots,a^2_1,a^2_0-1]$. Using
$(p_1-2)(p_1(2-q)+2q)=2p_1(p_1-2)-(p_1-2)^2 q$ and dividing by
$p_1^2\,q(p-q)$ yields
\[
-\frac{1}{[a^2_{n_2},\dots,a^2_1,a^2_0-1]}
+\frac{2(p_1-2)}{p_1(p-q)}
-\frac{(p_1-2)^2 q}{p_1^2(p-q)},
\]
and substituting into~(2) completes the proof.
\end{proof}

\section{The N\'{e}methi--Nicolaescu formula}\label{sec:nn}

In this section we show that the formula for $\theta(\xi_{\rm can})$
in Theorem~\ref{thm:main-seifert} agrees with the
N\'{e}methi--Nicolaescu expression
\[
K^2+\#\mathcal{V}-2,
\]
where $K$ is the canonical divisor on some resolution and
$\#\mathcal{V}$ is the number of irreducible components of the
exceptional divisor (equivalently, the number of vertices of
$\Gamma$). The identity
$\theta(\xi_{\rm can})=K^2+\#\mathcal{V}-2$ is due to
Gompf~\cite{Gompf98}; see also
\cite[Remark~4.8]{NemethiNicolaescu2002}. The closed-form expression
for $K^2+\#\mathcal{V}$ in the Seifert fibered case is given in
\cite[\S5]{NemethiNicolaescu2002}.

\begin{proposition}\label{prop:nemethi-agreement}
Under the hypotheses of Theorem~\ref{thm:main-seifert},
\[
\theta(\xi_{\rm can})
=
\frac{1}{e(Y)}\!\left(2-k+\sum_{i=1}^k\frac{1}{p_i}\right)^{\!2}
+e(Y)+3-12\sum_{i=1}^k s(q_i,p_i),
\]
where $s(q,p)$ denotes the classical Dedekind sum. After matching
conventions, this expression coincides with the
N\'{e}methi--Nicolaescu formula for $K^2+\#\mathcal{V}-2$ in
\cite[\S5]{NemethiNicolaescu2002}.
\end{proposition}

\begin{proof}
Recall from Theorem~\ref{thm:main-seifert} that
\[
\theta(\xi_{\rm can})
=
(2k-1)
-\sum_{i=1}^k\!\left(I(p_i/q_i)+\frac{q_i+q_i^*+2}{p_i}\right)
+\frac{1}{e(Y)}\!\left(-e_0+k-2-\sum_{i=1}^k\frac{q_i+1}{p_i}\right)^{\!2}.
\]
Set
\[
e:=e(Y)=e_0+\sum_{i=1}^k\frac{q_i}{p_i},
\qquad
S:=\sum_{i=1}^k\frac{1}{p_i},
\qquad
A:=2-k+S.
\]
Then
\[
-e_0+k-2-\sum_{i=1}^k\frac{q_i+1}{p_i}
=
-\!\left(e_0+\sum_{i=1}^k\frac{q_i}{p_i}\right)+k-2-\sum_{i=1}^k\frac{1}{p_i}
=
-(e+A),
\]
hence
\[
\frac{1}{e}\!\left(-e_0+k-2-\sum_{i=1}^k\frac{q_i+1}{p_i}\right)^{\!2}
=
\frac{(e+A)^2}{e}
=
e+2A+\frac{A^2}{e}.
\]
Substituting and splitting
$(q_i+q_i^*+2)/p_i=(q_i+q_i^*)/p_i+2/p_i$,
\[
\theta(\xi_{\rm can})
=
(2k-1)
-\sum_{i=1}^k\!\left(I(p_i/q_i)+\frac{q_i+q_i^*}{p_i}\right)
-2S+e+2A+\frac{A^2}{e}.
\]
Using $A=2-k+S$,
\[
(2k-1)-2S+e+2A
=
(2k-1)-2S+e+2(2-k+S)
=
e+3,
\]
so
\[
\theta(\xi_{\rm can})
=
\frac{A^2}{e}+e+3
-\sum_{i=1}^k\!\left(I(p_i/q_i)+\frac{q_i+q_i^*}{p_i}\right).
\]
Finally, the classical Hirzebruch--Zagier identity
\[
12\,s(q_i,p_i)
=
I(p_i/q_i)+\frac{q_i+q_i^*}{p_i}
\]
gives the stated formula.

To reconcile with the formula
in~\cite[\S5]{NemethiNicolaescu2002}, we match conventions explicitly.
N\'{e}methi--Nicolaescu work with the unnormalized Seifert invariants
$(\alpha_i,\beta_i)$ and pass to the normalized invariants
$(\alpha_i,\omega_i)$ via
\[
0\le\omega_i<\alpha_i,\qquad
\omega_i\equiv -\beta_i\pmod{\alpha_i},
\]
where the rational Euler number takes the form
$e=b+\sum_i\omega_i/\alpha_i$ for some integer $b$ with $b\le e<0$
(see~\cite[\S2.15]{NemethiNicolaescu2002}). The plumbing graph there
is star-shaped with central decoration $b$ and arms determined by the
continued fraction expansions
$\alpha_i/\omega_i=[b_{i1},\dots,b_{i\nu_i}]$.

Comparing to our notation, in which $r_i=q_i/p_i$ with $0<q_i<p_i$,
$e_0\in\Z$, and the legs are described by
$1/r_i=p_i/q_i=[a^i_1,\dots,a^i_{n_i}]$, the dictionary is
\[
\alpha_i=p_i,\qquad
\omega_i=q_i,\qquad
b=e_0,\qquad
b_{ij}=a^i_j,\qquad
\nu_i=n_i.
\]
(For $\omega_i=q_i$: both are the unique integer in $[0,\alpha_i)$
with the prescribed congruence class, and $0<q_i<p_i$ matches
$0\le\omega_i<\alpha_i$.) The relation between the two
Seifert--invariant conventions is therefore
$q_i\equiv-\beta_i\pmod{p_i}$. Since the classical Dedekind sum is
odd in its first argument, $s(\beta_i,\alpha_i)=s(-q_i,p_i)=-s(q_i,p_i)$,
so
\[
-12\sum_{i=1}^k s(q_i,p_i)
=
+12\sum_{i=1}^k s(\beta_i,\alpha_i),
\]
and the formula above becomes
\[
\theta(\xi_{\rm can})
=
\frac{1}{e(Y)}\!\left(2-k+\sum_{i=1}^k\frac{1}{\alpha_i}\right)^{\!2}
+e(Y)+3+12\sum_{i=1}^k s(\beta_i,\alpha_i).
\]
The right-hand side is precisely $K^2+\#\mathcal{V}-2$: the
closed-form expression for $K^2+\#\mathcal{V}$ is the
N\'{e}methi--Nicolaescu formula~\cite[\S5]{NemethiNicolaescu2002}, and
the identity $\theta(\xi_{\rm can})=K^2+\#\mathcal{V}-2$ is recorded
in~\cite[Remark~4.8]{NemethiNicolaescu2002}.
\end{proof}

\section{General negative-definite plumbing trees}\label{sec:general}

Let $\Gamma$ be a connected negative-definite tree with vertex set
$V(\Gamma)$, $N=|V(\Gamma)|$. Each vertex $v$ carries an integer
weight $-a_v\le -1$, and we set $z_v=a_v-2$.

\subsection*{Rooted tree structure}

Fix a root $\rho\in V(\Gamma)$. Since $\Gamma$ is a tree, there is a
unique path between any two vertices. We use this to define a partial
order on $V(\Gamma)$ as follows.

\begin{itemize}
\item The \dfn{depth} of a vertex $v$ is the number of edges in the
      unique path from $\rho$ to $v$.
\item The \dfn{parent} of a vertex $v\ne\rho$, denoted $\pi(v)$, is
      the unique neighbor of $v$ on the path from $v$ to $\rho$. The
      root $\rho$ has no parent.
\item The \dfn{children} of $v$, denoted $C(v)$, are the neighbors of
      $v$ at depth one greater than $v$:
      \[
      C(v)=\{u\in V(\Gamma)\mid \pi(u)=v\}.
      \]
      A vertex with $C(v)=\emptyset$ is called a \dfn{leaf}.
\item A \dfn{branching vertex} is a vertex $v$ with $|C(v)|\ge 2$.
\item The \dfn{rooted subtree} $\Gamma_v$ is the full subgraph of
      $\Gamma$ induced by $v$ together with all of its descendants.
      Equivalently,
      \[
      V(\Gamma_v)=\{u\in V(\Gamma):v\text{ lies on the path from }\rho\text{ to }u\},
      \]
      where this includes $u=v$ itself. In particular
      $\Gamma_\rho=\Gamma$, and $\Gamma_v=\{v\}$ when $v$ is a leaf.
      The vertex $v$ is the root of $\Gamma_v$, and the child subtrees
      $\Gamma_{c_1},\ldots,\Gamma_{c_s}$ for $c_1,\ldots,c_s\in C(v)$
      partition $V(\Gamma_v)\setminus\{v\}$.
\item A \dfn{topological sort} of $V(\Gamma)$ is any ordering
      $v_1,\dots,v_N$ of the vertices such that every child appears
      before its parent. Such an ordering always exists and allows
      the recursions below to be computed in a single pass from
      leaves to root.
\end{itemize}

Write $Q_v=Q_{\Gamma_v}$ for the intersection matrix of $\Gamma_v$,
with rows and columns ordered so that $v$ appears first, followed by
the vertices of each child subtree $\Gamma_c$ in turn. Each $Q_v$ is
a principal submatrix of $Q_\Gamma$, hence negative definite. The
rotation vector has components $z_v=a_v-2$ at each vertex $v$, and
$\mathbf{z}_v$ denotes its restriction to $V(\Gamma_v)$.

\subsection*{The recursive quantities}

We define three quantities at each vertex $v$, computed in a single
leaf-to-root pass.

\begin{itemize}
\item $\alpha_v=-(Q_v^{-1})_{vv}\in\Q_{>0}$: the negation of the
      $(v,v)$ diagonal entry of $Q_v^{-1}$. Since $Q_v$ is negative
      definite, all diagonal entries of $Q_v^{-1}$ are negative, so
      $\alpha_v>0$. Explicitly, $\alpha_v$ is the generalized
      continued fraction of the subtree $\Gamma_v$: for a leaf,
      $\alpha_v=1/a_v$, and for a linear chain rooted at $v$ with
      continued fraction $p/q=[a_v,\dots]$, one has $\alpha_v=q/p$.

\item $s_v=z_v-\sum_{c\in C(v)}\beta_c$: the \dfn{corrected rotation
      number} at $v$. It is the rotation number $z_v=a_v-2$ at $v$,
      corrected by the propagated contributions $\beta_c$ from each
      child subtree $\Gamma_c$. The sum is empty at leaves, so
      $s_v=z_v$ there.

\item $\beta_v=(Q_v^{-1}\mathbf{z}_v)_v=-\alpha_v s_v\in\Q$: the
      component at $v$ of the vector $Q_v^{-1}\mathbf{z}_v$. It
      measures how the rotation data of the subtree $\Gamma_v$
      propagates to its root $v$, and is the only information about
      $\Gamma_v$ that its parent needs. At a leaf,
      $\beta_v=-z_v/a_v$.
\end{itemize}

These three quantities satisfy the following recursions, computed
from leaves upward in any topological sort.

\begin{lemma}\label{lem:tree-recursion}
The quantities $\alpha_v$, $s_v$, $\beta_v$ satisfy the recursions
\begin{equation}\label{eq:alpha-recursion}
\alpha_v
=\frac{1}{a_v-\displaystyle\sum_{c\in C(v)}\alpha_c},
\end{equation}
\begin{equation}\label{eq:s-recursion}
s_v=z_v-\sum_{c\in C(v)}\beta_c,
\end{equation}
\begin{equation}\label{eq:beta-recursion}
\beta_v=-\alpha_v s_v,
\end{equation}
with base cases $\alpha_v=1/a_v$, $s_v=z_v$, $\beta_v=-z_v/a_v$ at
leaves. Moreover, for each vertex $v$:
\begin{equation}\label{eq:quadratic-recursion}
\mathbf{z}_v^TQ_v^{-1}\mathbf{z}_v
=
-\sum_{u\in V(\Gamma_v)}\alpha_u s_u^2.
\end{equation}
\end{lemma}

\begin{proof}
We proceed by induction on $|V(\Gamma_v)|$. The base case $v$ a leaf
is immediate: $Q_v=[-a_v]$, $Q_v^{-1}=[-1/a_v]$, so $\alpha_v=1/a_v$,
$s_v=z_v$, $\beta_v=-z_v/a_v$, and
$\mathbf{z}_v^TQ_v^{-1}\mathbf{z}_v=-z_v^2/a_v=-\alpha_v s_v^2$.

For the inductive step, let $c_1,\dots,c_s$ be the children of $v$.
Order the rows and columns of $Q_v$ so that $v$ comes first, followed
by the vertices of each child subtree $\Gamma_{c_i}$ in turn. Since
$\Gamma$ is a tree, vertices in different child subtrees are
non-adjacent, so $Q_v$ has the block form
\[
Q_v=\begin{pmatrix}-a_v & U \\ U^T & A\end{pmatrix},
\qquad
A=Q_{c_1}\oplus\cdots\oplus Q_{c_s},
\quad
U=\bigl(\mathbf{u}_{c_1}^T,\dots,\mathbf{u}_{c_s}^T\bigr),
\]
where $\mathbf{u}_c=(1,0,\dots,0)^T\in\R^{|V(\Gamma_c)|}$ selects
the root vertex of $Q_c$. Since $A$ is invertible (negative definite),
the Schur complement of $A$ in $Q_v$ is
\[
\delta_v
=-a_v-UA^{-1}U^T
=-a_v-\sum_{c\in C(v)}(Q_c^{-1})_{cc}
=-a_v+\sum_{c\in C(v)}\alpha_c.
\]
Negative definiteness of $Q_v$ forces $\delta_v<0$, equivalently
$a_v-\sum_c\alpha_c>0$, and
\[
\alpha_v=-(Q_v^{-1})_{vv}=-\frac{1}{\delta_v}
=\frac{1}{a_v-\sum_{c\in C(v)}\alpha_c},
\]
which proves \eqref{eq:alpha-recursion}. The block inverse of $Q_v$
is
\[
Q_v^{-1}=\begin{pmatrix}
\delta_v^{-1} & -\delta_v^{-1}UA^{-1}\\[1mm]
-A^{-1}U^T\delta_v^{-1} &
A^{-1}+A^{-1}U^T\delta_v^{-1}UA^{-1}
\end{pmatrix}.
\]
The component at $v$ of $Q_v^{-1}\mathbf{z}_v$ is
\[
\beta_v
=\delta_v^{-1}z_v-\delta_v^{-1}UA^{-1}\mathbf{z}_{\rm leg}
=\frac{1}{\delta_v}\!\left(z_v-\sum_{c\in C(v)}\beta_c\right)
=-\alpha_v s_v,
\]
since $UA^{-1}\mathbf{z}_{\rm leg}
=\sum_c\mathbf{u}_c^TQ_c^{-1}\mathbf{z}_c=\sum_c\beta_c$ and
$\delta_v^{-1}=-\alpha_v$. This proves \eqref{eq:s-recursion} and
\eqref{eq:beta-recursion}.

For \eqref{eq:quadratic-recursion}, Lemma~\ref{lem:schur-k}(iii)
applied to $Q_v$ gives
\[
\mathbf{z}_v^TQ_v^{-1}\mathbf{z}_v
=\sum_{c\in C(v)}\mathbf{z}_c^TQ_c^{-1}\mathbf{z}_c
+\frac{1}{\delta_v}\!\left(z_v-\sum_{c\in C(v)}\beta_c\right)^{\!2}
=\sum_{c\in C(v)}\mathbf{z}_c^TQ_c^{-1}\mathbf{z}_c
-\alpha_v s_v^2.
\]
Applying the inductive hypothesis
$\mathbf{z}_c^TQ_c^{-1}\mathbf{z}_c=-\sum_{u\in V(\Gamma_c)}\alpha_u s_u^2$
to each child and summing yields
\[
\mathbf{z}_v^TQ_v^{-1}\mathbf{z}_v
=-\sum_{c\in C(v)}\sum_{u\in V(\Gamma_c)}\alpha_u s_u^2
-\alpha_v s_v^2
=-\sum_{u\in V(\Gamma_v)}\alpha_u s_u^2,
\]
since $V(\Gamma_v)=\{v\}\cup\bigsqcup_{c\in C(v)}V(\Gamma_c)$.
\end{proof}

\begin{theorem}\label{thm:tree-general}
Let $\Gamma$ be a connected negative-definite tree with $N$ vertices,
and all genera zero. Let $Y=\partial X_\Gamma$ be the oriented
boundary of the plumbing $4$-manifold, and let $\xi_{\rm can}$ be the
canonical contact structure on $Y$. Fix any root $\rho\in V(\Gamma)$
and define $\alpha_v$, $s_v$, $\beta_v$ by the recursions
\eqref{eq:alpha-recursion}--\eqref{eq:beta-recursion} of
Lemma~\ref{lem:tree-recursion}. Then
\begin{equation}\label{eq:theta-tree}
\theta(\xi_{\rm can})
=
(N-2)
-\sum_{v\in V(\Gamma)}\alpha_v s_v^2.
\end{equation}
In particular, the right-hand side is independent of the choice of
root $\rho$.
\end{theorem}

\begin{proof}
Let $K_{\rm can}\in\operatorname{Char}(X_\Gamma)$ denote the canonical
characteristic vector, defined by $K_{\rm can}\cdot v+v\cdot v=-2$
for every $v\in V(\Gamma)$. In the basis dual to $\{[v]\}_{v\in V(\Gamma)}$,
the components of $K_{\rm can}$ are
$(K_{\rm can})_v=a_v-2=z_v$, so $z=K_{\rm can}$ in the dual basis,
and hence
\[
z^T Q_\Gamma^{-1}z=K_{\rm can}^2.
\]
Applying Lemma~\ref{lem:tree-recursion}\eqref{eq:quadratic-recursion}
to the root $\rho$, so that $\Gamma_\rho=\Gamma$, gives
\[
z^T Q_\Gamma^{-1}z=-\sum_{v\in V(\Gamma)}\alpha_v s_v^2.
\]
By~\cite[Remark~4.8]{NemethiNicolaescu2002},
\[
\theta(\xi_{\rm can})=K_{\rm can}^2+N-2,
\]
which combines with the previous two equations to give
\eqref{eq:theta-tree}. Independence of the root follows since
$\theta(\xi_{\rm can})$ is a topological invariant.
\end{proof}

\begin{example}\label{ex:two-branching}
We illustrate the computation of $\theta(\xi_{\rm can})$ on a
negative-definite tree $\Gamma$ with two branching vertices $v_2$
and $v_5$, depicted below.

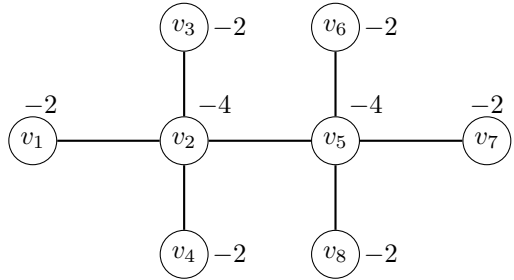
\begin{center}
\begin{tikzpicture}[
  every node/.style={circle, draw, fill=white, inner sep=2pt,
                     minimum size=18pt, font=\small},
  edge/.style={thick}
]
\node (v1) at (0,0)    {$v_1$};
\node (v2) at (2,0)    {$v_2$};
\node (v5) at (4,0)    {$v_5$};
\node (v7) at (6,0)    {$v_7$};
\node (v3) at (2, 1.5) {$v_3$};
\node (v4) at (2,-1.5) {$v_4$};
\node (v6) at (4, 1.5) {$v_6$};
\node (v8) at (4,-1.5) {$v_8$};
\draw[edge] (v1)--(v2);
\draw[edge] (v2)--(v5);
\draw[edge] (v5)--(v7);
\draw[edge] (v2)--(v3);
\draw[edge] (v2)--(v4);
\draw[edge] (v5)--(v6);
\draw[edge] (v5)--(v8);
\node[draw=none, fill=none] at (0.1,0.47)  {$-2$};
\node[draw=none, fill=none] at (2.4, 0.47) {$-4$};
\node[draw=none, fill=none] at (4.4, 0.47) {$-4$};
\node[draw=none, fill=none] at (6, 0.47)   {$-2$};
\node[draw=none, fill=none] at (2.6, 1.5)  {$-2$};
\node[draw=none, fill=none] at (2.6,-1.5)  {$-2$};
\node[draw=none, fill=none] at (4.6, 1.5)  {$-2$};
\node[draw=none, fill=none] at (4.6,-1.5)  {$-2$};
\end{tikzpicture}
\end{center}

The rotation numbers are $z_{v_2}=z_{v_5}=2$ and $z_v=0$ for all
other vertices. We apply Theorem~\ref{thm:tree-general} with two
different choices of root to illustrate root-independence.

\medskip
\noindent\textbf{Root $v_1$.} The parent--child structure is
$v_1\to\{v_2\}$, $v_2\to\{v_3,v_4,v_5\}$, $v_5\to\{v_6,v_7,v_8\}$. A
topological sort is $v_3,v_4,v_6,v_7,v_8,v_5,v_2,v_1$. At the leaves
$v_3,v_4,v_6,v_7,v_8$, each with $a_v=2$ and $z_v=0$, the recursions
give $\alpha_v=1/2$, $s_v=0$, $\beta_v=0$. At $v_5$ (children
$v_6,v_7,v_8$):
\[
s_{v_5}=2,\quad
\alpha_{v_5}=\tfrac{1}{4-3/2}=\tfrac{2}{5},\quad
\beta_{v_5}=-\tfrac{4}{5}.
\]
At $v_2$ (children $v_3,v_4,v_5$):
\[
s_{v_2}=2-(0+0-\tfrac{4}{5})=\tfrac{14}{5},\quad
\alpha_{v_2}=\tfrac{1}{4-1/2-1/2-2/5}=\tfrac{5}{13},\quad
\beta_{v_2}=-\tfrac{14}{13}.
\]
At the root $v_1$ (one child $v_2$):
\[
s_{v_1}=0-(-\tfrac{14}{13})=\tfrac{14}{13},\quad
\alpha_{v_1}=\tfrac{1}{2-5/13}=\tfrac{13}{21},\quad
\beta_{v_1}=-\tfrac{2}{3}.
\]

\medskip
\noindent\textbf{Root $v_5$.} Now the parent--child structure is
$v_5\to\{v_2,v_6,v_7,v_8\}$, $v_2\to\{v_1,v_3,v_4\}$, with
$v_1,v_3,v_4,v_6,v_7,v_8$ all leaves. The leaf computations are
unchanged. At $v_2$ (children $v_1,v_3,v_4$):
\[
s_{v_2}=2,\quad
\alpha_{v_2}=\tfrac{1}{4-1/2-1/2-1/2}=\tfrac{2}{5},\quad
\beta_{v_2}=-\tfrac{4}{5}.
\]
At the root $v_5$:
\[
s_{v_5}=2-(-\tfrac{4}{5})=\tfrac{14}{5},\quad
\alpha_{v_5}=\tfrac{1}{4-2/5-3/2}=\tfrac{10}{21},\quad
\beta_{v_5}=-\tfrac{4}{3}.
\]

\medskip
\noindent\textbf{Comparison.} Although the individual quantities
$\alpha_v,s_v,\beta_v$ depend on the choice of root, the
contributions $\alpha_v s_v^2$ summed over all vertices coincide:
\[
\renewcommand{\arraystretch}{1.3}
\begin{array}{c|cccccccc|c}
v & v_1 & v_2 & v_3 & v_4 & v_5 & v_6 & v_7 & v_8 & \sum_v \alpha_v s_v^2 \\
\hline
\text{rooted at } v_1: \alpha_v s_v^2 & \tfrac{28}{39} & \tfrac{196}{65} & 0 & 0 & \tfrac{8}{5} & 0 & 0 & 0 & \tfrac{16}{3} \\
\text{rooted at } v_5: \alpha_v s_v^2 & 0 & \tfrac{8}{5} & 0 & 0 & \tfrac{56}{15} & 0 & 0 & 0 & \tfrac{16}{3}
\end{array}
\]

Therefore
\[
\theta(\xi_{\rm can})=(N-2)-\sum_v\alpha_v s_v^2
=6-\tfrac{16}{3}=\tfrac{2}{3},
\]
in agreement with Theorem~\ref{thm:tree-general}.

\medskip
\noindent\textbf{The Heegaard Floer correction term.} Each
non-leaf vertex $v$ of $\Gamma$ satisfies $a_v = \deg(v)$, and each
leaf has $a_v = 2 \geq \deg(v) = 1$, so $\Gamma$ is rational and in
particular almost rational. By Remark~\ref{rem:d-invariant-tree},
\[
d(Y_\Gamma,\mathfrak{s}_{\rm can})
= \frac{\theta(\xi_{\rm can}) + 2}{4}
= \frac{2/3 + 2}{4}
= \frac{2}{3}.
\]
\end{example}

\begin{example}\label{ex:non-AR-trefoil}
Let $\Gamma$ be the tree depicted below with two bad vertices~\cite[Example~4.4.1]{Nemethi2008} (see
also~\cite[Figure~1, $n=2$]{OzsvathStipsiczSzabo2014knots}): 

\begin{center}
\begin{tikzpicture}[
  every node/.style={circle, draw, fill=white, inner sep=2pt,
                     minimum size=18pt, font=\small},
  edge/.style={thick}
]
\node (v2) at (-3, 0.7)  {$v_2$};
\node (v3) at (-3,-0.7)  {$v_3$};
\node (v1) at (-1.5, 0)  {$v_1$};
\node (v0) at (0,0)      {$v_0$};
\node (v4) at (1.5,0)    {$v_4$};
\node (v5) at (3, 0.7)   {$v_5$};
\node (v6) at (3,-0.7)   {$v_6$};
\draw[edge] (v2)--(v1); \draw[edge] (v3)--(v1); \draw[edge] (v1)--(v0);
\draw[edge] (v0)--(v4);
\draw[edge] (v4)--(v5); \draw[edge] (v4)--(v6);
\node[draw=none, fill=none] at (-3.0, 1.15) {$-2$};
\node[draw=none, fill=none] at (-3.0,-1.15) {$-3$};
\node[draw=none, fill=none] at (-1.5, 0.45) {$-1$};
\node[draw=none, fill=none] at ( 0.0, 0.45) {$-13$};
\node[draw=none, fill=none] at ( 1.5, 0.45) {$-1$};
\node[draw=none, fill=none] at ( 3.0, 1.15) {$-2$};
\node[draw=none, fill=none] at ( 3.0,-1.15) {$-3$};
\end{tikzpicture}
\end{center}

The matrix $Q_\Gamma$ is negative definite with $\det Q_\Gamma = -1$,
so $Y_\Gamma$ is an integer homology 3-sphere. Both $-1$-vertices
$v_1, v_4$ are bad ($\deg = 3 > 1 = a_v$), and modifying any single
weight cannot remove the badness of both, so $\Gamma$ is not almost
rational. The graph $\Gamma$ has two vertices of weight $-1$, but
Theorem~\ref{thm:tree-general} still applies.
 Rooting $\Gamma$ at $v_0$,
the leaf computations give $\alpha_{v_2}=1/2, s_{v_2}=0$;
$\alpha_{v_3}=1/3, s_{v_3}=1, \beta_{v_3}=-1/3$; and similarly on the
other arm. At the $-1$-vertex $v_1$ (children $v_2, v_3$),
\[
\alpha_{v_1}=\frac{1}{1-(\tfrac12+\tfrac13)} = 6, \qquad
s_{v_1} = -1 - (0 + (-\tfrac13)) = -\tfrac23, \qquad
\beta_{v_1} = -6\cdot(-\tfrac23) = 4,
\]
and similarly $\alpha_{v_4}=6, s_{v_4}=-2/3, \beta_{v_4}=4$. At the
root $v_0$ (with $a_{v_0}=13$, $z_{v_0}=11$, children $v_1, v_4$),
\[
\alpha_{v_0} = \frac{1}{13-12} = 1, \qquad s_{v_0} = 11 - (4+4) = 3.
\]
Theorem~\ref{thm:tree-general} therefore yields
\[
\theta(\xi_{\rm can}) = (N-2) - \sum_{v\in V(\Gamma)} \alpha_v s_v^2
= 5 - \bigl(\tfrac13 + \tfrac13 + \tfrac83 + \tfrac83 + 9\bigr)
= 5 - 15 = -10.
\]
\end{example}

\begin{example}\label{ex:non-AR}
We give an example of a tree $\Gamma$ for which the inequality in
Remark~\ref{rem:d-invariant-tree} is strict, illustrating that the
hypothesis of almost rationality there is essential. Let $\Gamma$ be
the tree below: two adjacent branching vertices $v_4, v_5$ of weight
$-2$, each carrying three leaves of weight $-4$.

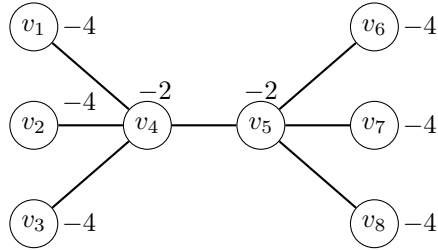
\begin{center}
\begin{tikzpicture}[
  every node/.style={circle, draw, fill=white, inner sep=2pt,
                     minimum size=18pt, font=\small},
  edge/.style={thick}
]
\node (v1) at (-1.5, 1.3) {$v_1$};
\node (v2) at (-1.5, 0)   {$v_2$};
\node (v3) at (-1.5,-1.3) {$v_3$};
\node (v4) at (0,0)       {$v_4$};
\node (v5) at (1.5,0)     {$v_5$};
\node (v6) at (3, 1.3)    {$v_6$};
\node (v7) at (3, 0)      {$v_7$};
\node (v8) at (3,-1.3)    {$v_8$};
\draw[edge] (v1)--(v4); \draw[edge] (v2)--(v4); \draw[edge] (v3)--(v4);
\draw[edge] (v4)--(v5);
\draw[edge] (v5)--(v6); \draw[edge] (v5)--(v7); \draw[edge] (v5)--(v8);
\node[draw=none, fill=none] at (-0.9, 1.3) {$-4$};
\node[draw=none, fill=none] at (-0.9, 0.3)   {$-4$};
\node[draw=none, fill=none] at (-0.9,-1.3) {$-4$};
\node[draw=none, fill=none] at (0.1,  0.45) {$-2$};
\node[draw=none, fill=none] at (1.5,  0.45) {$-2$};
\node[draw=none, fill=none] at (3.6,  1.3) {$-4$};
\node[draw=none, fill=none] at (3.6,  0)   {$-4$};
\node[draw=none, fill=none] at (3.6, -1.3) {$-4$};
\end{tikzpicture}
\end{center}

\medskip
\noindent\textbf{Negative-definiteness.}  $Q_\Gamma$ is negative definite by Sylvester's
criterion. In particular $\det Q_\Gamma = 2304 \neq 0$, so
$Y_\Gamma$ is a rational homology 3-sphere.

\medskip
\noindent\textbf{Failure of almost rationality.} Both branching
vertices are bad, since $a_{v_4}=a_{v_5}=2 < 4 = \deg(v_4)=\deg(v_5)$.
Modifying the weight at any single vertex of $\Gamma$ cannot remove
the badness of both $v_4$ and $v_5$ simultaneously, so $\Gamma$ is not
almost rational in the sense of \cite{Nemethi2005}.

\medskip
\noindent\textbf{The $\theta$-invariant.} Rooting $\Gamma$ at $v_4$,
the recursions of Lemma~\ref{lem:tree-recursion} give
$\alpha_v = 1/4$ and $s_v = 2$ at each leaf. At $v_5$ (children
$v_6, v_7, v_8$),
\[
\alpha_{v_5} = \frac{1}{2 - 3 \cdot \tfrac{1}{4}} = \frac{4}{5},
\qquad
s_{v_5} = 0 - 3\cdot\bigl(-\tfrac{1}{2}\bigr) = \frac{3}{2},
\qquad
\beta_{v_5} = -\frac{6}{5},
\]
and at the root $v_4$ (children $v_1, v_2, v_3, v_5$),
\[
\alpha_{v_4} = \frac{1}{2 - \bigl(\tfrac{3}{4}+\tfrac{4}{5}\bigr)} = \frac{20}{9},
\qquad
s_{v_4} = 0 - \bigl(3\cdot(-\tfrac{1}{2}) + (-\tfrac{6}{5})\bigr) = \frac{27}{10}.
\]
Theorem~\ref{thm:tree-general} therefore yields
\[
\theta(\xi_{\rm can})
\;=\;
(N-2) - \sum_{v \in V(\Gamma)} \alpha_v s_v^2
\;=\;
6 - \Bigl(6 \cdot \tfrac{1}{4} \cdot 4 + \tfrac{4}{5}\cdot\tfrac{9}{4} + \tfrac{20}{9}\cdot\tfrac{729}{100}\Bigr)
\;=\;
6 - 24
\;=\;
-18.
\]

\medskip
\noindent\textbf{The canonical characteristic vector.} The vector
$K_{\rm can} \in \operatorname{Char}(X_\Gamma)$ is determined by
$K_{\rm can} \cdot v + v\cdot v = -2$ for every $v$, so in the basis
dual to $\{[v_i]\}$ the components are $(K_{\rm can})_v = a_v - 2$, namely
\[
K_{\rm can} = (2,2,2,0,0,2,2,2)^T.
\]
By the $v_4 \leftrightarrow v_5$ symmetry of $\Gamma$, the vector
$Q_\Gamma^{-1} K_{\rm can}$ has the form
$(\alpha,\alpha,\alpha,\beta,\beta,\alpha,\alpha,\alpha)^T$,
where $\alpha,\beta$ satisfy the leaf and branching equations
$-4\alpha + \beta = 2$ and $3\alpha - \beta = 0$, giving $\alpha=-2$,
$\beta=-6$. Hence
\[
Q_\Gamma^{-1} K_{\rm can} = (-2,-2,-2,-6,-6,-2,-2,-2)^T,
\qquad
K_{\rm can}^2 = K_{\rm can}^T Q_\Gamma^{-1} K_{\rm can} = 6 \cdot 2 \cdot (-2) = -24,
\]
consistent with the identity $\theta(\xi_{\rm can}) = K_{\rm can}^2 + N - 2$.

\medskip
\noindent\textbf{The maximum of $k^2$ over $\mathfrak{s}_{\rm can}$.}
Characteristic vectors representing $\mathfrak{s}_{\rm can}$ are
exactly
\[
k = K_{\rm can} + 2 Q_\Gamma\, c, \qquad c \in \mathbb{Z}^N,
\]
(see e.g.~\cite[eq.~(3)]{Khan2026}). Expanding,
\[
k^2 = K_{\rm can}^2 + 4\,c^T K_{\rm can} + 4\,c^T Q_\Gamma\, c,
\]
which is strictly concave in $c$ since the Hessian $8 Q_\Gamma$ is
negative definite. Hence $k^2$ attains a unique maximum on
$\mathbb{R}^N$ at the critical point
\[
c^* = -\tfrac{1}{2} Q_\Gamma^{-1} K_{\rm can},
\]
and we have an upper bound $\max_{c \in \mathbb{Z}^N} k^2 \leq k^2(c^*)$,
with equality if and only if $c^* \in \mathbb{Z}^N$. For a generic
negative-definite plumbing, $c^*$ is not integral and the integer
maximum is strictly smaller. In the present example, however, the
$v_4 \leftrightarrow v_5$ symmetry computed above gives
$Q_\Gamma^{-1} K_{\rm can} = (-2,-2,-2,-6,-6,-2,-2,-2)^T$, so
\[
c^* = (1,1,1,3,3,1,1,1)^T \in \mathbb{Z}^N.
\]
The integer maximum therefore coincides with the real maximum, and
the corresponding characteristic vector is
\[
k = K_{\rm can} + 2 Q_\Gamma c^*
= K_{\rm can} - Q_\Gamma Q_\Gamma^{-1} K_{\rm can}
= 0,
\]
giving $\max_{k \in \mathfrak{s}_{\rm can}} k^2 = 0$.

\medskip
\noindent\textbf{Conclusion.} By Khan's theorem~\cite{Khan2026},
\[
d(Y_\Gamma, \mathfrak{s}_{\rm can}) = \frac{0+8}{4} = 2,
\qquad\text{whereas}\qquad
\frac{\theta(\xi_{\rm can})+2}{4} = \frac{-18+2}{4} = -4.
\]
The inequality of Remark~\ref{rem:d-invariant-tree} is therefore
strict, with gap
\[
d(Y_\Gamma,\mathfrak{s}_{\rm can}) - \frac{\theta(\xi_{\rm can})+2}{4}
= \frac{1}{4}\bigl(\max_{k\in\mathfrak{s}_{\rm can}} k^2 - K_{\rm can}^2\bigr)
= 6.
\]
\end{example}

\begin{corollary}\label{cor:star-recovery}
Theorem~\ref{thm:tree-general} reduces to Theorem~\ref{thm:main-seifert}
when $\Gamma$ is star-shaped.
\end{corollary}

\begin{proof}
Root $\Gamma$ at the central vertex $\rho$ of weight $e_0$, so
$z_\rho=-e_0-2$. The children of $\rho$ are the first vertices
$v_1^i$ of each leg $i=1,\dots,k$, and each leg is a linear chain
with weights $-a_1^i,\dots,-a_{n_i}^i$ and continued fraction
$p_i/q_i=[a_1^i,\dots,a_{n_i}^i]$.

\medskip
\noindent\textbf{The $\alpha$-recursion along leg $i$.}
Applying \eqref{eq:alpha-recursion} from the tip of leg $i$ inward
gives $\alpha_{v_1^i}=q_i/p_i=r_i$, the standard continued fraction of
the leg. This is consistent with $(Q_i^{-1})_{v_1^iv_1^i}=-r_i$ as
used in Lemma~\ref{lem:schur-k}.

\medskip
\noindent\textbf{The $\beta$-recursion along leg $i$.}
Applying \eqref{eq:beta-recursion} from the tip inward gives
$\beta_{v_1^i}=(q_i+1-p_i)/p_i$ by Lemma~\ref{lem:leg}.

\medskip
\noindent\textbf{Contribution of leg $i$ to the sum.}
By \eqref{eq:quadratic-recursion} applied to $\Gamma_{v_1^i}$,
\[
-\sum_{u\in V(\Gamma_{v_1^i})}\alpha_u s_u^2
=\mathbf{z}_{v_1^i}^TQ_i^{-1}\mathbf{z}_{v_1^i}
=\mathbf{w}_i^TA_i^{-1}\mathbf{w}_i
=2-I(p_i/q_i)-n_i-\frac{q_i+q_i^*+2}{p_i},
\]
where the last equality is Lemma~\ref{lem:leg}.

\medskip
\noindent\textbf{At the root $\rho$.}
The corrected rotation number at $\rho$ is
\[
s_\rho
=z_\rho-\sum_{i=1}^k\beta_{v_1^i}
=(-e_0-2)-\sum_{i=1}^k\frac{q_i+1-p_i}{p_i}
=-e_0+k-2-\sum_{i=1}^k\frac{q_i+1}{p_i},
\]
and
\[
\alpha_\rho
=\frac{1}{-e_0-\sum_{i=1}^k r_i}
=\frac{-1}{e(Y)},
\]
so the root contributes
\[
-\alpha_\rho s_\rho^2
=\frac{1}{e(Y)}
\!\left(-e_0+k-2-\sum_{i=1}^k\frac{q_i+1}{p_i}\right)^{\!2}.
\]

\medskip
\noindent\textbf{Assembling.}
Since $N=1+\sum_i n_i$, we have $N-2=\sum_i n_i-1$. Therefore
\begin{align*}
\theta(\xi_{\rm can})
&=(N-2)-\sum_{v\in V(\Gamma)}\alpha_v s_v^2\\
&=\!\left(\sum_{i=1}^k n_i-1\right)
+\sum_{i=1}^k\!\left(2-I(p_i/q_i)-n_i-\frac{q_i+q_i^*+2}{p_i}\right)\\
&\quad
+\frac{1}{e(Y)}\!\left(-e_0+k-2-\sum_{i=1}^k\frac{q_i+1}{p_i}\right)^{\!2}\\
&=(2k-1)
-\sum_{i=1}^k\!\left(I(p_i/q_i)+\frac{q_i+q_i^*+2}{p_i}\right)
+\frac{1}{e(Y)}\!\left(-e_0+k-2-\sum_{i=1}^k\frac{q_i+1}{p_i}\right)^{\!2},
\end{align*}
which is exactly Theorem~\ref{thm:main-seifert}.
\end{proof}

\section{$\theta$-minimization and rational homology ball fillings}\label{sec:rhb}

In this final section we use Theorem~\ref{thm:consistent} to obtain a
strict $\theta$-minimization property of $\xi_{\rm can}$ among
all diagram-realizable contact structures, and apply this to rule out
symplectic rational homology ball fillings on a large class of Stein
fillable contact rational homology $3$-spheres.

The following observation is well known
\cite[Lemma~2.14]{EtnyreOzbagciTosun2025}.

\begin{lemma}\label{lem:rhb}
Let $(Y,\xi)$ be an oriented rational homology $3$-sphere equipped
with a contact structure $\xi$. If $\theta(\xi)\neq -2$, then
$(Y,\xi)$ is not symplectically fillable by a rational homology ball.
\end{lemma}

We briefly recall the symplectic plumbing trees introduced by
Stipsicz, Szab\'{o} and Wahl~\cite{StipsiczSzaboWahl08}.

\begin{definition}[\cite{StipsiczSzaboWahl08}, Definition~1.1]
A plumbing tree $\Gamma$ on $N$ vertices is a \emph{symplectic
plumbing tree} if $\Gamma$ admits an embedding
\[
\varphi\colon\Gamma\hookrightarrow(\Z^N,Q_N)
\]
into the negative-definite diagonal lattice $(\Z^N,Q_N)$ with
$Q_N=N\langle -1\rangle$, such that:
\begin{itemize}
\item For vertices $v_1\ne v_2\in\Gamma$,
    \[
    Q_N(\varphi(v_1),\varphi(v_2))=
    \begin{cases}
    1 & \text{if }v_1\text{ and }v_2\text{ are adjacent in }\Gamma,\\
    0 & \text{otherwise,}
    \end{cases}
    \]
\item $Q_N(\varphi(v),\varphi(v))$ equals the decoration of $v$, for
      all $v\in\Gamma$,
\item With the basis $\{E_1,\dots,E_N\}$ of $Q_n$ satisfying
      $Q_N(E_i,E_j)=-\delta_{ij}$ and setting $K=\sum_{i=1}^N E_i$,
      the \emph{adjunction equality}
      \[
      Q_N(\varphi(v),K)+Q_N(\varphi(v),\varphi(v))=-2 
      \]
      holds for every vertex $v$ of $\Gamma$.
\end{itemize}
A plumbing tree $\Gamma$ is called \emph{minimal} if there is no
vertex in $\Gamma$ with decoration $-1$. We denote the set of minimal
connected symplectic plumbing trees by $\mathcal{S}$.
\end{definition}

\begin{theorem}[{\cite[Corollary~2.5]{StipsiczSzaboWahl08}}]
If a complex surface singularity admits a rational homology disk
($\Q$HD) smoothing, then its resolution graph $\Gamma$ belongs to
$\mathcal{S}$.
\end{theorem}

\begin{remark}
Note that if a complex surface singularity admits a $\Q$HD smoothing,
then the Milnor fiber, equipped with its Stein structure, provides a
symplectic rational homology ball filling of the singularity link
$Y_\Gamma$ equipped with its canonical contact structure
$\xi_{\rm can}$, where $\Gamma$ is the resolution graph of the
singularity.
\end{remark}

\begin{lemma}\label{lem:symp}
If $\Gamma\in\mathcal{S}$, then $\theta(\xi_{\rm can})=-2$.
\end{lemma}

\begin{proof}
Let $\Gamma\in\mathcal{S}$ and suppose that $\Gamma$ has $N$ vertices.
If $Z_K$ is the canonical cycle, then $\varphi(Z_K)=-K$ and it follows
that
\[
Z_K\cdot Z_K=Q_N(\varphi(Z_K),\varphi(Z_K))=Q_N(K,K)=-N,
\]
since $K=\sum_{i=1}^N E_i$ and $Q_N(E_i,E_j)=-\delta_{ij}$. Let
$X_\Gamma$ denote the canonical $4$-manifold obtained by the plumbing
graph $\Gamma$. Using the facts that $\sigma(X_\Gamma)=-N$ and
$\chi(X_\Gamma)=N+1$, we obtain
\begin{equation}\label{eq-theta}
\begin{split}
\theta(\xi_{\rm can})
&=c_1^2(X_\Gamma,J)-3\sigma(X_\Gamma)-2\chi(X_\Gamma)\\
&=Z_K\cdot Z_K-3(-N)-2(N+1)=-N+3N-2N-2=-2,
\end{split}
\end{equation}
since $c_1^2(X_\Gamma,J)=Z_K\cdot Z_K$.
\end{proof}

By Lemma~\ref{lem:symp}, $\theta(\xi_{\rm can})=-2$ for each
$\Gamma\in\mathcal{S}$, and $\mathcal{S}$ decomposes as the union
\[
\mathcal{G}\cup\mathcal{W}\cup\mathcal{N}\cup\mathcal{M}\cup\mathcal{A}\cup\mathcal{B}\cup\mathcal{C}
\]
of seven explicit families as shown by Stipsicz, Szab\'{o} and
Wahl~\cite{StipsiczSzaboWahl08}. They also showed that if $\Gamma$
belongs to $\mathcal{G}\cup\mathcal{W}\cup\mathcal{N}\cup\mathcal{M}$,
then the corresponding singularity admits a $\Q$HD smoothing, which
implies that $(Y_\Gamma,\xi_{\rm can})$ has a rational homology
ball symplectic filling. However, this result holds only for some of the plumbing trees in
$\mathcal{A}\cup\mathcal{B}\cup\mathcal{C}$, as shown
in~\cite{BhupalStipsicz11} (see also~\cite{Beke25}). 

We now turn to \emph{inconsistent} contact structures on $Y_\Gamma$,
where $\Gamma$ is described as in the introduction. The inequality in
Proposition~\ref{prop:minimum} below was proved
in~\cite[Proposition~1.7]{EtnyreOzbagciTosun2025} for \emph{all
tight} contact structures on $Y_\Gamma$, where $\Gamma$ is a
star-shaped plumbing tree with three legs. The proof there relies on
the classification of tight contact structures on $Y_\Gamma$, but
such a classification is not available for a general
negative-definite plumbing $\Gamma$. So here we only prove the
inequality for diagram-realizable contact structures, which are
inconsistent.

\begin{proposition}\label{prop:minimum}
The strict inequality
\[
\theta(\xi_{\rm can})<\theta(\xi)
\]
holds for any inconsistent diagram-realizable contact structure $\xi$
on $Y_\Gamma$.
\end{proposition}

\begin{proof}
For each vertex $u\in V(\Gamma)$, let $-a_u$ denote the decoration
(self-intersection) at $u$, so $a_u\geq 2$ by minimality of $\Gamma$.
A diagram-realizable contact structure on $Y_\Gamma$ is determined
(up to isotopy) by a choice of rotation number $\rot(L_u)$ for each
Legendrian unknot $L_u$ in the surgery diagram. Since
$\tb(L_u)=1-a_u$, we have $|\rot(L_u)|\leq a_u-2$ with
$\rot(L_u)\equiv a_u-2\pmod{2}$. Collecting these, we regard the
rotation data as a vector
\[
\mathbf{z}=(z_u)_{u\in V(\Gamma)},\qquad z_u:=\rot(L_u),
\]
subject to $|z_u|\leq a_u-2$. The \emph{maximal} rotation vector is
\[
\mathbf{c}:=(a_u-2)_{u\in V(\Gamma)};
\]
so the consistent diagrams correspond to $\mathbf{z}=\pm\mathbf{c}$
(all rotation numbers at the maximum in absolute value, with a
uniform sign), and inconsistent diagrams are those with
$\mathbf{z}\neq\pm\mathbf{c}$.

By Theorem~\ref{thm:consistent}, the canonical contact structure
$\xi_{\rm can}$ is realized by the consistent assignment, i.e.,
by rotation vector $\pm\mathbf{c}$. Since $c_1^2$ is quadratic in the
rotation vector, the two choices $+\mathbf{c}$ and $-\mathbf{c}$ give
the same value of $\theta$.

Let $\xi$ be an inconsistent diagram-realizable contact structure
with rotation vector $\mathbf{z}$; so $|z_u|\leq a_u-2$ for all $u$
and $\mathbf{z}\neq\pm\mathbf{c}$. Both $\xi$ and
$\xi_{\rm can}$ are induced by Stein structures on the same
smooth plumbing $X_\Gamma$, by Gompf's
theorem~\cite[Theorem~2.1]{Gompf98}. Hence the topological term
$2\chi(X_\Gamma)+3\sigma(X_\Gamma)$ is the same for both, and
\[
\theta(\xi)-\theta(\xi_{\rm can})
=\mathbf{z}^TQ_\Gamma^{-1}\mathbf{z}-\mathbf{c}^TQ_\Gamma^{-1}\mathbf{c},
\]
where $Q_\Gamma$ is the intersection matrix of $\Gamma$.

Since $\Gamma$ is connected and negative definite, $-Q_\Gamma$ is an
irreducible Stieltjes matrix, so
\[
M:=-Q_\Gamma^{-1}
\]
has strictly positive entries. Thus
\[
\theta(\xi)-\theta(\xi_{\rm can})
=\mathbf{c}^TM\mathbf{c}-\mathbf{z}^TM\mathbf{z}.
\]

Since every entry of $M$ is positive,
\[
\mathbf{z}^TM\mathbf{z}
=\sum_{u,v}M_{uv}z_u z_v
\leq\sum_{u,v}M_{uv}|z_u||z_v|
=|\mathbf{z}|^TM|\mathbf{z}|,
\]
and since $|z_u|\leq c_u=a_u-2$ for every $u$,
\[
|\mathbf{z}|^TM|\mathbf{z}|\leq\mathbf{c}^TM\mathbf{c}.
\]
Therefore $\mathbf{z}^TM\mathbf{z}\leq\mathbf{c}^TM\mathbf{c}$.

Equality in the first inequality holds only when all nonzero $z_u$
have the same sign. Equality in the second inequality holds only
when $|z_u|=a_u-2$ for every $u$. Together, equality throughout
holds precisely when $\mathbf{z}=\pm\mathbf{c}$.

Since $\xi$ is inconsistent, $\mathbf{z}\neq\pm\mathbf{c}$, and hence
\[
\mathbf{z}^TM\mathbf{z}<\mathbf{c}^TM\mathbf{c}.
\]
Equivalently, $\theta(\xi)-\theta(\xi_{\rm can})>0$, so
\[
\theta(\xi_{\rm can})<\theta(\xi).
\]
\end{proof}

\begin{corollary}\label{cor:d-strict}
Let $\Gamma$ be a minimal connected negative-definite plumbing tree with all 
vertices of genus zero, and assume that $\Gamma$ is almost rational. Let $\xi$ 
be an inconsistent diagram-realizable contact structure on $Y_\Gamma$, and let 
$\mathfrak{s}_\xi \in \mathrm{Spin}^c(Y_\Gamma)$ denote the Spin$^c$ structure 
induced by $\xi$. Then
\[
d(Y_\Gamma, \mathfrak{s}_{\mathrm{can}}) < d(Y_\Gamma, \mathfrak{s}_\xi).
\]
\end{corollary}

\begin{proof}
The contact structure $\xi$ is 
the contact boundary of the Stein structure $J_\xi$ on the plumbing 4-manifold 
$X_\Gamma$ obtained from the Legendrian-surgery diagram realizing the rotation 
vector of $\xi$. Let $K_\xi := c_1(X_\Gamma, J_\xi) \in H^2(X_\Gamma; \mathbb{Z})$ 
denote the corresponding characteristic vector, satisfying 
$\langle K_\xi, [C_v]\rangle = \mathrm{rot}(L_v)$ for every vertex $v$. Since 
$X_\Gamma$ is a negative-definite 4-manifold with $b_2(X_\Gamma) = N$ and 
$\partial X_\Gamma = Y_\Gamma$, the Ozsv\'ath--Szab\'o inequality 
\cite[Theorem~9.6]{OzsvathSzabo2003abs}, applied to the Spin$^c$ structure 
$\mathfrak{s}_{J_\xi}$ on $X_\Gamma$ underlying $J_\xi$, gives
\begin{equation}\label{eq:OS-bound}
\frac{K_\xi^2 + N}{4} \;\leq\; d(Y_\Gamma, \mathfrak{s}_\xi),
\end{equation}
where $\mathfrak{s}_\xi := \mathfrak{s}_{J_\xi}|_{Y_\Gamma}$. By Gompf's 
handlebody formula,
\begin{equation}\label{eq:theta-K}
\theta(\xi) \;=\; K_\xi^2 + N - 2,
\end{equation}
so \eqref{eq:OS-bound} can be rewritten as 
$\frac{\theta(\xi) + 2}{4} \leq d(Y_\Gamma, \mathfrak{s}_\xi)$. 
By Proposition~\ref{prop:minimum},
\begin{equation}\label{eq:theta-strict}
\theta(\xi_{\mathrm{can}}) \;<\; \theta(\xi).
\end{equation}
Since $\Gamma$ is almost rational, $K_{\mathrm{can}}$ realizes 
$\max_{k \in \mathfrak{s}_{\mathrm{can}}} k^2$ by 
\cite[Theorem~6.3]{Nemethi2005}, and Khan's theorem \cite{Khan2026} then gives 
equality in \eqref{eq:OS-bound} for the canonical Spin$^c$ structure 
(cf.\ Remark~1.6):
\begin{equation}\label{eq:can-equality}
d(Y_\Gamma, \mathfrak{s}_{\mathrm{can}}) 
\;=\; \frac{K_{\mathrm{can}}^2 + N}{4}
\;=\; \frac{\theta(\xi_{\mathrm{can}}) + 2}{4}.
\end{equation}
Combining \eqref{eq:can-equality}, \eqref{eq:theta-strict}, and 
\eqref{eq:OS-bound},
\[
d(Y_\Gamma, \mathfrak{s}_{\mathrm{can}}) 
\;=\; \frac{\theta(\xi_{\mathrm{can}}) + 2}{4} 
\;<\; \frac{\theta(\xi) + 2}{4} 
\;\leq\; d(Y_\Gamma, \mathfrak{s}_\xi). \qedhere
\]
\end{proof}

\begin{definition}
We denote by $\mathcal{T}$ the set of all minimal connected
negative-definite plumbing trees $\Gamma$ such that
$\theta(\xi_{\rm can})=-2$.
\end{definition}

By Lemma~\ref{lem:symp}, $\mathcal{S}$ is a subset of $\mathcal{T}$.
In fact it is a proper subset: for $n\geq 2$ the star-shaped tree
$\mathrm{fpp}(n)$ of \cite[Example~5.1]{Beke25} has a node of valency
$n^2+n+1\geq 7$, hence does not belong to any of the seven families
in $\mathcal{S}$ (in which all nodes have valency at most $4$), so
$\mathrm{fpp}(n)\notin\mathcal{S}$. On the other hand
$\mathrm{fpp}(n)\in\mathcal{T}$ by \cite[Theorem~4.1]{Beke25}. See Section~\ref{sec:examples} for more examples of negative-definite plumbing trees in $\mathcal{T} \setminus \mathcal{S}$.

\begin{corollary}\label{cor:S}
Suppose that $\Gamma\in\mathcal{T}$ and $\xi$ is a diagram-realizable
contact structure on $Y_\Gamma$. If $\xi$ is inconsistent, then
$(Y_\Gamma,\xi)$ does not have a rational homology ball symplectic
filling.
\end{corollary}

\begin{proof}
Combining Theorem~\ref{thm:consistent} and
Proposition~\ref{prop:minimum}, for any inconsistent contact
structure $\xi$ on $Y_\Gamma$ we have the inequality
\[
-2=\theta(\xi_{\rm can})<\theta(\xi),
\]
which proves that $(Y_\Gamma,\xi)$ does not have a rational homology
ball symplectic filling, by Lemma~\ref{lem:rhb}.
\end{proof}

\section{Some examples with $\theta=-2$}\label{sec:examples}

In this section we provide more examples of negative-definite plumbing trees in $\mathcal{T} \setminus \mathcal{S}$.

\begin{lemma}\label{lem:theta-three-families}
Let $\ell \ge 1$, and let $\Gamma$ be a star-shaped plumbing graph with $k$ legs, each leg a linear chain of length $\ell$, all vertices having framing $-2$. Let the central vertex have framing $-b$. Then the canonical contact structure $\xi_{\rm can}$ on the Seifert fibered space $Y_\Gamma$ satisfies $\theta(\xi_{\rm can})=-2$ in each of the following cases:
\begin{enumerate}
\item[(1)] $k=4\ell+4$ and $b=k=4\ell+4$,
\item[(2)] $k=\ell^2+3\ell+3$ and $b=k+1=\ell^2+3\ell+4$,
\item[(3)] $k=4\ell^2-3$ and $b=k+4=4\ell^2+1$.
\end{enumerate}
\end{lemma}

\begin{proof}
Since each leg is a linear chain of length $\ell$ with all framings equal to $-2$, we have
\[
[2,\dots,2]=\frac{\ell+1}{\ell}.
\]
Thus for every leg
\[
p=\ell+1,\qquad q=\ell,\qquad q^*=\ell,\qquad I(p/q)=-\ell.
\]
Substituting into the general formula for $\theta(\xi_{\rm can})$ given in Theorem~\ref{thm:main-seifert}, we obtain
\[
\theta(\xi_{\rm can})
=
k\ell-1+\frac{(b-2)^2}{-b+\frac{k\ell}{\ell+1}}.
\]
Therefore $\theta(\xi_{\rm can})=-2$ is equivalent to
\begin{equation}\label{eq:theta-condition}
(\ell+1)(b-2)^2=(k\ell+1)\big((\ell+1)b-k\ell\big).
\end{equation}
We verify \eqref{eq:theta-condition} in each case.

\medskip

\noindent
\textbf{Case (1):} $k=4\ell+4$, $b=4\ell+4$.

\[
b-2=4\ell+2=2(2\ell+1),
\quad
k\ell+1=4\ell^2+4\ell+1=(2\ell+1)^2,
\]
\[
(\ell+1)b-k\ell=4(\ell+1).
\]
Hence both sides of \eqref{eq:theta-condition} equal
\[
4(\ell+1)(2\ell+1)^2.
\]

\medskip

\noindent
\textbf{Case (2):} $k=\ell^2+3\ell+3$, $b=\ell^2+3\ell+4$.

\[
b-2=(\ell+1)(\ell+2),
\quad
k\ell+1=(\ell+1)^3,
\quad
(\ell+1)b-k\ell=(\ell+2)^2.
\]
Hence both sides of \eqref{eq:theta-condition} equal
\[
(\ell+1)^3(\ell+2)^2.
\]

\medskip

\noindent
\textbf{Case (3):} $k=4\ell^2-3$, $b=4\ell^2+1$.

\[
b-2=4\ell^2-1=(2\ell-1)(2\ell+1),
\]
\[
k\ell+1=4\ell^3-3\ell+1=(\ell+1)(2\ell-1)^2,
\]
\[
(\ell+1)b-k\ell=(2\ell+1)^2.
\]
Hence both sides of \eqref{eq:theta-condition} equal
\[
(\ell+1)(2\ell-1)^2(2\ell+1)^2.
\]

Thus \eqref{eq:theta-condition} holds in all three cases, and therefore $\theta(\xi_{\rm can})=-2$.
\end{proof}

\begin{corollary}\label{cor:SEF} 
 None of the Seifert fibered spaces $Y_\Gamma$  listed in Lemma~\ref{lem:theta-three-families} admit sympletic rational homology ball fillings.
 \end{corollary}

\begin{proof} Let $\Gamma$ be any star-shaped plumbing graph  listed   in Lemma~\ref{lem:theta-three-families}. All symplectically fillable contact structures on the Seifert fibered space $Y_\Gamma$ are diagram-realizable by 
~\cite{CavalloMatkovic2026b}. The canonical contact structure does not admit a symplectic rational homology ball filling by ~\cite[Theorem 1.4]{BhupalStipsicz11}. None of the other diagram-realizable contact structures admit symplectic  rational homology ball fillings either, by Corollary~\ref{cor:S} and Theorem~\ref{thm:consistent}. 
\end{proof}

\begin{lemma}\label{lem:theta-sporadic-examples}
Let $\Gamma$ be a star-shaped plumbing graph with $k$ legs, each leg a linear chain of length $\ell$, all vertices having framing $-2$, and let the central vertex have framing $-b$. Then the canonical contact structure $\xi_{\rm can}$ on the Seifert fibered space $Y_\Gamma$ satisfies $\theta(\xi_{\rm can})=-2$ in each of the following cases:
\begin{enumerate}
\item[(1)] $\ell=1$, $k=4$, $b=7$,
\item[(2)] $\ell=2$, $k=3$, $b=9$,
\item[(3)] $\ell=1$, $k=7$, $b=4$,
\item[(4)] $\ell=1$, $k=8$, $b=5$,
\item[(5)] $\ell=1$, $k=9$, $b=7$.
\end{enumerate}
Moreover, none of these examples is covered by Lemma~\ref{lem:theta-three-families}.
\end{lemma}

\begin{proof}
As in the proof of Lemma~\ref{lem:theta-three-families}, since each leg is a chain of length $\ell$ with all framings equal to $-2$, we have
\[
[2,\dots,2]=\frac{\ell+1}{\ell},
\]
so that
\[
p=\ell+1,\qquad q=\ell,\qquad q^*=\ell,\qquad I(p/q)=-\ell.
\]
Hence the general formula for $\theta(\xi_{\rm can})$ given in Theorem~\ref{thm:main-seifert} becomes
\[
\theta(\xi_{\rm can})
=
k\ell-1+\frac{(b-2)^2}{-b+\frac{k\ell}{\ell+1}}.
\]
Therefore $\theta(\xi_{\rm can})=-2$ is equivalent to
\begin{equation}\label{eq:sporadic-theta-condition}
(\ell+1)(b-2)^2=(k\ell+1)\big((\ell+1)b-k\ell\big).
\end{equation}
We now verify \eqref{eq:sporadic-theta-condition} in each case.

\medskip

\noindent
\textbf{Case (1):} $(\ell,k,b)=(1,4,7)$.

\[
(\ell+1)(b-2)^2=2\cdot 5^2=50,
\]
and
\[
(k\ell+1)\big((\ell+1)b-k\ell\big)
=(4+1)(2\cdot 7-4)=5\cdot 10=50.
\]
Thus \eqref{eq:sporadic-theta-condition} holds.

\medskip

\noindent
\textbf{Case (2):} $(\ell,k,b)=(2,3,9)$.

\[
(\ell+1)(b-2)^2=3\cdot 7^2=147,
\]
and
\[
(k\ell+1)\big((\ell+1)b-k\ell\big)
=(3\cdot 2+1)(3\cdot 9-3\cdot 2)=7\cdot 21=147.
\]
Thus \eqref{eq:sporadic-theta-condition} holds.

\medskip

\noindent
\textbf{Case (3):} $(\ell,k,b)=(1,7,4)$.

\[
(\ell+1)(b-2)^2=2\cdot 2^2=8,
\]
and
\[
(k\ell+1)\big((\ell+1)b-k\ell\big)
=(7+1)(2\cdot 4-7)=8\cdot 1=8.
\]
Thus \eqref{eq:sporadic-theta-condition} holds.

\medskip

\noindent
\textbf{Case (4):} $(\ell,k,b)=(1,8,5)$.

\[
(\ell+1)(b-2)^2=2\cdot 3^2=18,
\]
and
\[
(k\ell+1)\big((\ell+1)b-k\ell\big)
=(8+1)(2\cdot 5-8)=9\cdot 2=18.
\]
Thus \eqref{eq:sporadic-theta-condition} holds.

\medskip

\noindent
\textbf{Case (5):} $(\ell,k,b)=(1,9,7)$.

\[
(\ell+1)(b-2)^2=2\cdot 5^2=50,
\]
and
\[
(k\ell+1)\big((\ell+1)b-k\ell\big)
=(9+1)(2\cdot 7-9)=10\cdot 5=50.
\]
Thus \eqref{eq:sporadic-theta-condition} holds.

Therefore $\theta(\xi_{\rm can})=-2$ in all five cases.

Finally, none of these examples appears in Lemma~\ref{lem:theta-three-families}. Indeed, for $\ell=1$ that lemma gives only
\[
(k,b)=(8,8),\ (7,8),\ (1,5),
\]
and for $\ell=2$ it gives only
\[
(k,b)=(12,12),\ (13,14),\ (13,17).
\]
These do not include any of the five triples listed above.
\end{proof}

We recover Beke's star-shaped "$\mathrm{fpp}(n)$" trees \cite[Example~5.1]{Beke25}  below. 

\begin{corollary}
Let $n>1$, and let $\Gamma_n$ be the star-shaped plumbing graph with
\[
k=n^2+n+1
\]
legs, each of length $n-1$, all framings $-2$, and central framing $-(n^2+n+2)$. Then $\theta(\xi_{\rm can})=-2$.
\end{corollary}

\begin{proof}
This is case {\rm (2)} of Lemma~\ref{lem:theta-three-families} with $\ell=n-1$.
\end{proof}

\begin{lemma}
The smallest star-shaped plumbing graph with $k\ge 3$ legs, each leg a linear chain of length $\ell\ge 1$ with all framings equal to $-2$, and with central framing $-b$, for which the canonical contact structure satisfies $\theta(\xi_{\rm can})=-2$, is given by
\[
\ell=1,\qquad k=4,\qquad b=7.
\]
In particular, the smallest such graph has $5$ vertices.
\end{lemma}

\begin{proof}
For these symmetric graphs, the condition $\theta(\xi_{\rm can})=-2$ is equivalent to
\begin{equation}\label{eq:smallest-example}
(\ell+1)(b-2)^2=(k\ell+1)\big((\ell+1)b-k\ell\big).
\end{equation}
The total number of vertices is $1+k\ell$. Since $k\ge 3$ and $\ell\ge 1$, the smallest possible number of vertices is $4$, which corresponds to $(k,\ell)=(3,1)$.

In that case, \eqref{eq:smallest-example} becomes
\[
2(b-2)^2=4(2b-3),
\]
or equivalently
\[
b^2-8b+10=0.
\]
Its discriminant is $24$, which is not a square, so there is no integer solution $b$. Hence there is no example with $4$ vertices.

The next possibility is $1+k\ell=5$, which forces $(k,\ell)=(4,1)$. Then \eqref{eq:smallest-example} becomes
\[
2(b-2)^2=5(2b-4).
\]
For $b=7$, both sides are equal to $50$. Therefore $\theta(\xi_{\rm can})=-2$ for $(\ell,k,b)=(1,4,7)$.

Thus the smallest such example has $5$ vertices.
\end{proof}

\bibliography{references}
\bibliographystyle{plain}

\end{document}